\documentclass[11.5pt]{amsart}
\usepackage{graphicx,amssymb,amsmath,amsthm}
\usepackage{comment}
\usepackage{enumerate}
\usepackage{dsfont}
 \numberwithin{equation}{section}
\usepackage{mathrsfs}
\usepackage{epsfig}
\usepackage{lscape}
\usepackage{subfigure}
\usepackage{graphicx}
\usepackage{epstopdf}
\usepackage{color}
\usepackage{caption}
\usepackage{algorithm}
\usepackage{algpseudocode}
\usepackage{bm}
\usepackage{graphicx}
\usepackage[colorlinks,
            linkcolor=blue,
            anchorcolor=blue,
            citecolor=blue,
            urlcolor=blue,
            breaklinks=true
            ]{hyperref}

\newcommand{\DOI}[1]{DOI: \href{https://doi.org/#1}{#1}}

\allowdisplaybreaks[4]

\newcommand{\abs}[1]{\lvert#1\rvert}

\newcommand{\ve}{{\boldsymbol \eta}}

\renewcommand{\omega}{\ve}

\newcommand{\RNum}[1]{\uppercase\expandafter{\romannumeral #1\relax}}

\newtheorem{theorem}{Theorem}[section]
\newtheorem{lemma}{Lemma}[section]
\newtheorem{corollary}{Corollary}[section]

\newtheorem{remark}{Remark}[section]

\numberwithin{equation}{section}

\date{}

\begin{document}

\baselineskip 18pt

\title{Stability of Least Squares Approximation under Random Sampling}

\author{Zhiqiang Xu}
\thanks{Zhiqiang Xu is supported  by the National Science Fund for Distinguished Young Scholars (12025108), National Natural Science Foundation of China (12471361, 12021001, 12288201) and National Key R\&D Program of China (2023YFA1009401).}
\address{State Key Laboratory of Mathematical Sciences, Academy of Mathematics and Systems Science, Chinese Academy of Sciences, Beijing 100190, China; School of Mathematical Sciences, University of Chinese Academy of Sciences, Beijing, 100049, China}
\email{xuzq@lsec.cc.ac.cn}

\author{Xinyue Zhang}
\address{State Key Laboratory of Mathematical Sciences, Academy of Mathematics and Systems Science, Chinese Academy of Sciences, Beijing 100190, China; School of Mathematical Sciences, University of Chinese Academy of Sciences, Beijing, 100049, China}
\email{Corressponding author: zhangxinyue@amss.ac.cn}

 \begin{abstract}
    This paper investigates the stability of the least squares approximation $P_m^n$ within the univariate polynomial space of degree $m$, denoted by ${\mathbb P}_m$. The approximation $P_m^n$ entails identifying  a polynomial in ${\mathbb P}_m$ that approximates a function $f$ over a domain $X$ based on samples of $f$ taken at $n$ randomly selected points, according to a specified probability measure $\rho_X$.            
    The primary goal is to determine the sampling rate necessary to ensure the stability of $P_m^n$.
    Assuming the sampling points are i.i.d. with respect to a Jacobi weight function, we present the sampling rate that guarantee the stability of $P_m^n$.
    Specifically, for uniform random sampling, we demonstrate that a sampling rate of $n \asymp m^2$ is required to maintain stability. By combining these findings with those of Cohen-Davenport-Leviatan, we conclude that, for uniform random sampling, the optimal sampling rate for guaranteeing the stability of $P_m^n$ is $n \asymp m^2$, up to a $\log n$ factor. 
    Motivated by this result, we extend the impossibility theorem, previously applicable to equally spaced samples,   to the case of random samples, illustrating the balance between accuracy and stability in recovering analytic functions. 
    
 \end{abstract}



\maketitle

\section{Introduction}
\subsection{Problem setup} We consider the classical problem of approximating an unknown function $f$ on $X:=[-1,1]$ from its evaluations at the random sampling points $\{x_i\}_{i=1}^{n} \subset X$, where the $\{x_i\}_{i=1}^n$ are independent and identically distributed (i.i.d.) with respect to a given probability measure $\rho_X$. Such problem, especially when $f$ is a multivariate function, comes from the field of uncertainty quantification (UQ) \cite{sullivan2015introduction,xiu2009fast} where the input parameters used to describe the computational model are considered as random variables. For specific approximation methods, we concentrate on the polynomial discrete least squares. In the framework of UQ, this approach serves as a tool to build the generalized polynomial chaos expansion, aiming to expand the solution in polynomials of the input random variables \cite{eldred2009recent}. It is also known as polynomial regression \cite{draper1998applied}, widely used in optimal experimental design \cite{brook2018applied,jobson2012applied}, statistical learning theory \cite{vapnik1998statistical} and PDEs with stochastic data \cite{migliorati2013polynomial,migliorati2013approximation,chkifa2015discrete}. 

Specifically, let $\mathbb{P}_{m}$ be the space of univariate algebraic polynomials of degree at most $m$. The best approximation in $\mathbb{P}_{m}$ to a function $f$, measured by the $L^2(X,\rho_X)$ norm $\|f\| := \left( \int_X {\abs{f(x)}^2} {\rm d}\rho_X(x) \right)^{1/2}$, is
$$
P_m(f) := \mathop{\arg\min}\limits_{p \in \mathbb{P}_{m}}{\|p-f\|}.
$$
And its discrete least squares approach takes the form
\begin{equation}
    \label{least_squares}
    P_m^n(f) := \mathop{\arg\min}\limits_{p \in \mathbb{P}_{m}}{\|p-f\|_{n,2}},
\end{equation}
where $\|f\|_{n,2} := \left(\frac{1}{n} \sum_{i=1}^n{|f(x_i)|^2}\right)^{1/2}$ is the empirical $L^2(X,\rho_X)$ semi-norm with the associated empirical semi-inner product $\left<f,g\right>_n := \frac{1}{n} \sum_{i=1}^n{f(x_i)\overline{g(x_i)}}$, 
i.e., the operator $P_m^n$ is the orthogonal projector onto $\mathbb{P}_{m}$ with respect to the empirical semi-inner product $\left<\cdot,\cdot\right>_n$. We assume $m < n$ to ensure the almost certain uniqueness of the solution.

In this paper, our primary focus lies in examining the stability of $P_m^n$ under random sampling as  $m, n \rightarrow \infty$. This stability is quantified using the condition number, as originally defined in \cite{platte2011impossibility}, but with the norms replaced. Specifically, the condition number $\kappa_{2}(P_m^n)$ is given by:
\begin{equation}
	\label{cond_2}
	\kappa_{2}\left(P_m^n\right):=\sup _{f:X \rightarrow \mathbb{R}} \lim _{\delta \rightarrow 0^{+}} \sup _{\substack{h:X \rightarrow \mathbb{R} \\ 0<\|h\|_{n, 2} \leq \delta}} \frac{\left\|P_m^n(f+h)-P_m^n(f)\right\|}{\|h\|_{n, 2}} = \sup _{\substack{h:X \rightarrow \mathbb{R} \\ \|h\|_{n, 2} \neq 0}} \frac{\left\|P_m^n(h)\right\|}{\|h\|_{n, 2}},
\end{equation}
where the second equality is due to the linearity of the operator $P_m^n$.
We say that the discrete least squares approximation $P_m^n$ is {\em stable} if and only if $\kappa_{2}\left(P_m^n\right)$ can be bounded by some constants independent of $m$ and $n$, with high probability.

It is worth noting that in equation \eqref{cond_2}, the value of $\kappa_{2}\left(P_m^n\right)$ is invariant under replacing the supremum over all real-valued functions $f$ and $h$ by that over all complex-valued ones. This is because the stability of $P_m^n$ is independent of the particular function space to which $f$ and $h$ belong, provided that this space contains the approximation space $\mathbb{P}_{m}$. As evidenced in Lemma \ref{eq} below, it depends only on the following factors: the approximation space $\mathbb{P}_{m}$, the point set $\{x_i\}_{i=1}^n$, and the sampling rate, that is, the relation between the number of sampling points and the dimension of $\mathbb{P}_m$.

\subsubsection{Impact of sampling rate on stability}
A lot of work has been done on the stability of  $P_m^n$  \cite{migliorati2014analysis,cohen2013stability,migliorati2013approximation,migliorati2013polynomial,migliorati2015convergence,chkifa2015discrete}. Cohen, Davenport and Leviatan \cite{cohen2013stability} consider the domain $X\subset \mathbb{R}^d$ and the general linear subspace $V_m$ of $L^2(X,\rho_X)$, with ${\rm dim}(V_m)=m\leq n$, as the approximation space instead of $\mathbb{P}_m$, and gives the sufficient sampling rates to keep $P_m^n$ stable with high probability. 
Suppose that $\{L_j\}_{j=1}^{m}$ is an orthonormal basis of $V_m$ in the sense of $L^2(X,\rho_X)$.
We can express $P_m^n(f)$ where $f:X\rightarrow \mathbb{R}$ as:
$$
P_m^n(f) = \sum_{j=1}^{m}{u_j L_j},
$$
where $\mathbf{u} = (u_1,\dotsc,u_m)^\top$ is the solution of the $m \times m$ system
\begin{equation}
    \label{system}
    \mathbf{Gu=b}
\end{equation}
with $\mathbf{G}:=(\left<L_j, L_k\right>_n)_{j,k=1,\dotsc,m}$ and $\mathbf{b}:=(\left<f, L_k\right>_n)_{k=1,\dotsc,m}$. 
The results on stability in \cite{cohen2013stability} is obtained by a probabilistic estimate of the proximity of the
matrices $\mathbf{G}$ and $\mathbf{I}$ in spectral norm $|||\cdot|||$.
\begin{theorem}[{\cite[Theorem 1]{cohen2013stability}}]
	\label{cohen}
    Let $X$ be a domain of $\mathbb{R}^d$ and $\rho_X$ be a probability measure on $X$. Given a set of random samples $\{x_i\}_{i=1}^{n}\subset X$ i.i.d. with respect to $\rho_X$, and a linear subspace $V_m$ of $L^2(X,\rho_X)$ with ${\rm dim}(V_m)=m \leq n$, consider the discrete least squares projection $P_m^n$ as in \eqref{least_squares} with $\mathbb{P}_m$ replaced by $V_m$.
    Let
	\begin{equation}
		\label{quantity}
		K(m):=\sup_{x \in X} \sum_{j=1}^{m}{|L_j(x)|^2}
	\end{equation}
	where $\{L_j\}_{j=1}^{m}$ is an orthonormal basis of $V_m$ in the sense of $L^2(X,\rho_X)$. Then for any $r > 0$, if $m$ is such that
	\begin{equation*}
		K(m) \leq \iota \frac{n}{\log n}, \quad \text { with } \quad \iota:=\frac{1-\log 2}{2+2 r},
	\end{equation*}
	we have
	$$
    P\left(|||\mathbf{G}-\mathbf{I} ||| \leq \frac{1}{2}\right) \geq 1-2 n^{-r}
    $$
    where $\mathbf{G}$ is the matrix as in \eqref{system}, and for any function $f:X\rightarrow \mathbb{R}$, 
    $$
    \|P_m^n(f)\| \leq \sqrt{6}\;\|f\|_{n,2}
    $$
	with probability at least $1-2n^{-r}$.
\end{theorem}

The value of $K(m)$ is solely determined by the choice of the space $V_m$ and the measure $\rho_X$, as mentioned in \cite{cohen2013stability}. Now we turn to the case of $X=[-1,1]$, $V_{m+1} = \mathbb{P}_{m}$ and consider the uniform random sampling (${\rm d}\rho_X=\frac{1}{2}{\rm d}x$). Let $L_k$ denote the Legendre polynomial of degree $k-1$ with normalization. As is well-known, ${L}_k$, $k=1,\ldots, m+1$ forms an orthonormal basis of $\mathbb{P}_{m}$ in the $L^2(X,\rho_X)$ sense. Additionally, we have $\|L_k\|_\infty = |L_k(1)| = \sqrt{2k-1}$ for $k=1,\dotsc,m+1$. Thus, 
$
K(m+1) = \sum_{k=1}^{m+1}(2k-1) = (m+1)^2.
$
Then we simply deduce from Theorem \ref{cohen} that when 
$$
n \asymp m^2 \quad \text{as} \;\; n \rightarrow \infty,
$$
disregarding both constant factors and factors of $\log n$, the \emph{condition number} for $P_m^n$ as in \eqref{cond_2} satisfies 
$$
\kappa_2(P_m^n) \lesssim 1
$$ 
with high probability. 
Throughout this paper, we utilize the asymptotic notation $a_n \asymp b_n$ as $n \rightarrow \infty$ if there exist positive constants $c_1$ and $c_2$ such that $c_1 b_n \leq a_n \leq c_2 b_n$ holds for all large $n$. And we adopt the notation $A \lesssim B$ to denote that there exists a universal constant $c>0$ such that $A \leq cB$. We say that an event holds with high probability  if it occurs with probability at least $1 - C/n^s$ for some positive constants $C$ and $s$. 

Meanwhile, numerical experiments observe that when $n$ is less than $m^2$ in an asymptotic sense (or the sampling rate for alternative sampling measures is below a certain threshold), the least squares approximation may exhibit ill-conditioning and ultimately yield divergence \cite{migliorati2014analysis,boyd2009divergence}. 
However, it is important to note that although these empirical findings provide insights into the required sampling rates for achieving stability, there is still a lack of solid theoretical guarantees in this regard.  This gives rise to the central question that this article aims to address:

\textbf{Question \uppercase\expandafter{\romannumeral1}.} $\,$  {\em
For uniform random sampling, can we potentially reduce the sampling rate $n \asymp m^2$ while still ensuring stability of $P_m^n$ with high probability? In other words, can we guarantee the stability of $P_m^n$ with high probability even when the number of the uniform random sampling points is significantly less than $m^2$? 
}

Our focus is to establish the requisite conditions for random sampling rates under various sampling probability measures, to ensure the stability of plain polynomial least squares. Specifically, we investigate the so-called {\em Jacobi} weight functions (see  \cite{adcock2019optimal}), where the uniform measure serves as a special case. This investigation is motivated by the impossibility theorem outlined below.

Here, it's worth noting that to achieve nearly linear sampling complexity, numerous studies have employed the \emph{change of measure} strategy, that is, make the sampling measure denoted as $\rho_S$ different from the measure $\rho_T$ in which we evaluate the $L^2$ approximation error. See for instance \cite{cohen2017optimal,migliorati2019adaptive,adcock2020near,dolbeault2024randomized} where samples are drawn randomly from well-designed probability measures $\rho_S$ for any $\rho_T$, in conjunction with weighted least-squares method. Furthermore, to overcome the drawback of uniform quadratic oversampling required for the algebraic polynomial approximation space, some research endeavors \cite{adcock2023fast,adcock2020approximating} leverage polynomials over an extended interval $[-\gamma,\gamma]$, $\gamma>1$, to create an approximation on [-1,1], which is known as \emph{polynomial frame} approximation; alternative studies focus on constructing an orthonormal basis from spaces such as Sobolev space \cite{bartel2022error}, instead of the conventional usage of Legendre polynomials. Also, nearly linear oversampling is attainable by these approaches, albeit at a trade-off concerning the error decay rate. However, the aforementioned topics lie beyond the scope of our present paper.

\subsubsection{Impossibility theorem}
Setting $n=m$ in \eqref{least_squares}, least squares degenerates into polynomial interpolation, which is widely recognized to be exponential ill-conditioning at deterministic equally spaced points even when $f$ is analytic, given that the corresponding Lebesgue constants are approximately $2^n$ \cite{turetskii1940bounding}. It leads to divergence in floating-point arithmetic despite the theoretical convergence, which is commonly referred to as \emph{Runge’s phenomenon}. Numerous approximation methods have been investigated to mitigate the Runge effect, aiming to enhance accuracy and stability \cite{fasshauer2012stable,bos2012lebesgue}. In addition,  Platte, Trefethen and  Kuijlaars \cite{platte2011impossibility} offer valuable insights into the trade-off between these two factors, which is commonly referred to as the {\em impossibility theorem}. It is established in the context of deterministic sampling on $X=[-1,1]$ and the unknown function $f:X \rightarrow \mathbb{C}$ is analytic.

Specifically, consider an \emph{approximation procedure} $\{\phi_n\}_{n=1}^\infty$ defined as a family of mappings $\phi_n: C(X) \rightarrow C(X)$ for $1 \leq n < \infty$, where for any $f \in C(X)$, $\phi_n(f)$ depends solely on $\{f(x_i)\}_{i=1}^{n}$ with $\{x_i\}_{i=1}^{n}$ being the sampling nodes on $X$. The stability of $\{\phi_n\}_{n=1}^\infty$ that the impossibility theorem concerns is represented by the \emph{condition number}
\begin{equation}
    \label{cond_inf}
    \kappa_{\infty}\left(\phi_n\right):=\sup _{f \in C(X)} \lim _{\delta \rightarrow 0^{+}} \sup _{\substack{h \in C(X) \\ 0<\|h\|_{n, \infty} \leq \delta}} \frac{\left\|\phi_n(f+h)-\phi_n(f)\right\|_{\infty}}{\|h\|_{n, \infty}},
\end{equation}
where $\|\cdot\|_\infty$ is the supremum ($L^\infty(X)$) norm over $X$ and $\|f\|_{n,\infty} := \max_{i \in \{1,\dotsc,n\}}{\abs{f(x_i)}}$ is the discrete supremum semi-norm of a function $f$ on the sampling points. It simplifies to the Lebesgue constant if $\phi_n$ is linear.
When we regard the least squares approximation $P_m^n$ as $\phi_n$, as previously mentioned, the definition of the condition number $\kappa_2(P_m^n)$ in equation \eqref{cond_2} is derived from the general definition \eqref{cond_inf}, with the replacement of norms.
 Given a compact set $E$ in the complex plane, the Banach space $B(E)$ is defined as a space of functions that are continuous on $E$ and analytic in its interior. The norm $\|f\|_E$ of a function $f$ in $B(E)$ is given by $\sup_{z \in E}{\abs{f(z)}}$. Now, let's delve into the impossibility theorem.
 
\begin{theorem}[{\cite[Theorem 3.4]{platte2011impossibility}}]
	\label{impossibility}
	Let $E \subseteq \mathbb{C}$ be a compact set containing $X:=[-1,1]$ in its interior. Suppose that $\{\phi_n\}_{n=1}^\infty$ is an approximation procedure based on the equispaced grid of $n$ points $\{x_i\}_{i=1}^{n} = \{-1+\frac{2(i-1)}{n-1}\}_{i=1}^{n}\subset X$ such that, for some $M < \infty$, $\sigma>1$ and $1/2<\tau \leq 1$, 
	$$
	\left\|f-\phi_n(f)\right\|_{\infty} \leq M \sigma^{-n^\tau}\|f\|_E, \quad 1 \leq n < \infty,
	$$
    for all $f \in B(E)$. Then the condition numbers for $\phi_n$ satisfy
	$$
	\kappa_{\infty}\left(\phi_n\right) \geq C^{n^{2 \tau-1}},
	$$
	for some constant $C>1$ and all large $n$.
\end{theorem}
The theorem states that when approximating analytic functions, no approximation method using deterministic equispaced samples can maintain both stability and achieve error convergence beyond the root-exponential rate. Additionally, it implies that achieving exponential convergence in approximation inevitably leads to exponential instability.
Then, the work of Adcock, Platte, and Shadrin \cite{adcock2019optimal} extends this result to the deterministic equidistributed sampling points $\{x_i\}_{i=1}^{n}$ on $X=[-1,1]$, with respect to the {\em modified Jacobi} weight functions given by: 
$$
\mu(x) = g(x)(1-x)^\alpha(1+x)^\beta,\quad \text{with} \;\; \int_{-1}^{1}\mu(x){\rm d}x = 1, \;\; x \in X,
$$
where $\alpha, \beta > -1$ and $g(x) \in L^\infty(X)$ satisfies the condition $c_1 \leq g(x) \leq c_2$ almost everywhere. And these sampling points are implicitly defined by the equation:
\begin{equation*}
	\label{equidistributed}
	\frac{i-1}{n-1} = \int_{-1}^{x_i}\mu(x){\rm d}x, \quad i = 1,\dotsc,n.
\end{equation*}

\begin{theorem}[{\cite[Theorem 4.2]{adcock2019optimal}}]
	\label{impossibility_4}
	Consider a set of deterministic samples $\{x_i\}_{i=1}^{n} \subset X:=[-1,1]$, where $x_i$, $i=1,\dotsc,n$, are deterministically equidistributed with respect to a modified Jacobi weight function with parameters $\alpha,\beta>-1$ and $c_1\leq g(x)\leq c_2$ almost everywhere. Let $E \subseteq \mathbb{C}$ be a compact set containing $X$ in its interior. Suppose that $\left\{\phi_n\right\}_{n=1}^{\infty}$ is an approximation procedure based on $\{x_i\}_{i=1}^{n}$ such that, for some $M<\infty$, $\sigma>1$ and $\tau >0$,
	\begin{equation*}
		\left\|f-\phi_n(f)\right\|_{\infty} \leq M \sigma^{-n^\tau}\|f\|_E, \quad 1 \leq n < \infty,
	\end{equation*}
	for all $f \in B(E)$. If 
	$
	\gamma = \max\{\alpha, \beta\} > -1/2,
	$ 
	and 
	$$
	\tau > \frac{1}{2(1+\gamma)},
	$$
	then the condition numbers for $\phi_n$ satisfy
	\begin{equation*}
        \kappa_{\infty}\left(\phi_n\right) \geq C^{n^{r}}, \quad r = \frac{2(1+\gamma)\tau-1}{1+2\gamma},
	\end{equation*}
	{for some constant $C>1$} and all large $n$.
\end{theorem}
When deterministic uniform sampling ($\alpha, \beta = 0, g(x) = \frac{1}{2}$) is considered, Theorem \ref{impossibility_4} aligns with Theorem \ref{impossibility}. 
These findings are intriguing as they highlight a crucial observation: for any approximation method $\phi_n$ under specific sampling conditions, there exists an optimal convergence rate, such as $\mathcal{O}(\sigma^{-n^{\frac{1}{2(1+\gamma)}}})$ as outlined in Theorem \ref{impossibility_4}, when approximating analytic functions, provided that $\phi_n$ maintains stability. And in fact, this implies the necessary sampling rates for the stability of least squares, which will be further elucidated in the subsequent proof. The impossibility theorem is also instructive in recent work on deep learning \cite{colbrook2022difficulty} and other various approximation methods \cite{huybrechs2023aaa,veettil2022multivariate}. 

The aforementioned theorems are grounded in deterministic sampling. Nevertheless, random sampling is a common scenario as well. Consequently, we pose the following question to explore this matter:

\textbf{ Question \uppercase\expandafter{\romannumeral2}.} $\,$  {\em Does the impossibility theorem hold true under random sampling (from a probabilistic perspective)? }

\subsection{Our contribution}

The objective of this paper is to address both Question \uppercase\expandafter{\romannumeral1} and Question \uppercase\expandafter{\romannumeral2}. Assuming the sampling interval to be $X := [-1,1]$ and the sampling points $\{x_i\}_{i=1}^{n}$ to be i.i.d. according to a given probability measure $\rho_X$ on $X$, we proceed to present the key findings of this study.

\subsubsection{Answer for Question \uppercase\expandafter{\romannumeral1}}
We present the necessity of the uniform random sampling rate $n \asymp m^2$ by analyzing the stability of $P_m^n$ related to the sampling rates. In fact, we offer results under more general random sampling cases. Specifically, we focus on some clustering characteristics exhibited by random sampling near the endpoints, which are weaker than the quadratic clustering characteristics of the Chebyshev grids. This is given by the sampling probability measures $\rho_X$ with respect to the \emph{Jacobi} weight functions: 
\begin{equation}
    \label{Jacobi_e}
    w(x) = c(1-x)^\alpha(1+x)^\beta, \quad \text{with} \;\; c = \left(\int_{-1}^{1}(1-x)^\alpha(1+x)^\beta {\rm d}x\right)^{-1},\;\;  x \in X,
\end{equation}
where $\alpha, \beta > -1$, satisfying ${\rm d}\rho_X = w(x){\rm d}x$. It  covers uniform measure (${\rm d}\rho_X = \frac{1}{2}{\rm d}x$) and Chebyshev measure (${\rm d}\rho_X = \frac{1}{\pi\sqrt{1-x^2}}{\rm d}x$) as particular cases. 

Let $\gamma = \max\{\alpha, \beta\} > -1/2$. The essence of Theorem \ref{opt_rates} can be summarized as follows: if the sampling rate falls below $n \asymp m^{2(1+\gamma)}$, the exponential ill-conditioning is almost certain to occur, indicating that the sampling rate cannot be lower than $n \asymp m^{2(1+\gamma)}$ if we want the approximation process $P_m^n$ to remain stable with a high probability. Therefore, the necessity of choosing $n \asymp m^{2(1+\gamma)}$ is proven.

We deem it necessary to emphasize that in the case of uniform random sampling, where $\gamma=\alpha=\beta=0$, Theorem \ref{opt_rates} reveals that the sampling rate $n \asymp m^2$ is requisite for ensuring the stability of $P_m^n$. This provides an {\em answer} to Question I.

\begin{theorem}
    \label{opt_rates}
    Consider a set of random samples $\{x_i\}_{i=1}^{n}\subset X=[-1,1]$, where $x_i$, $i=1,\dotsc,n$, are i.i.d. with respect to a Jacobi weight function characterized by parameters $\alpha,\beta>-1$ (see \eqref{Jacobi_e}). Let $P_m^n$ be the discrete polynomial least squares projection as in \eqref{least_squares}. Assume that 
    \begin{equation}
        \label{sampling_re}
        n \asymp m^{1/\tau}, \quad 0 < \tau \leq 1,
    \end{equation}
    with $m < n$. If $\gamma := \max\{\alpha, \beta\} > -1/2$, and 
    $$
    \tau > \frac{1}{2(1+\gamma)},
    $$
    then for any $s \in \mathbb{R}$ satisfying $0<s<2(1+\gamma)\tau-1$ and all large $n$, the following holds with a probability of at least $1-\frac{2e^2\bar{c}}{n^s}$,
    \begin{equation*}
        \label{least_squares_mian_result}
        \kappa_{2}\left(P_m^n\right) \,\geq \,\frac{\eta}{n^{(1+\gamma)\tau}} \nu^{n^r}, \quad r = \frac{2(1+\gamma)\tau-1-s}{1+2\gamma}.
    \end{equation*}
    Here, $\bar{c} = \frac{c \cdot 2^{\min\{\alpha, \beta\}}}{1+\gamma}$
     where $c$ is the constant from \eqref{Jacobi_e}, and  $\eta>0$, $\nu>1$ are constants depending on $\alpha, \beta$ and 
      and the implicit constants in \eqref{sampling_re}.
\end{theorem}

By combining Theorem \ref{opt_rates} and Theorem \ref{cohen}, we further substantiate that the sampling rate $n \asymp m^{2(1+\gamma)}$ is both sufficient and necessary for the stability of $P_m^n$ with a high probability, if we disregard a factor of $\log n$.

\begin{corollary}
    \label{opt_sam_rat}
    Consider a set of random samples $\{x_i\}_{i=1}^{n}\subset X=[-1,1]$, where $x_i$, $i=1,\dotsc,n$, are i.i.d. with respect to a Jacobi weight function characterized by parameters $\alpha,\beta>-1$. Let $P_m^n$ be the discrete polynomial least squares projection as in \eqref{least_squares}. Assuming 
    $
    \gamma := \max\{\alpha, \beta\} > -1/2,
    $
    the sampling rate
    \begin{equation}
        \label{iff_rate}
        n \asymp m^{2(1+\gamma)}
    \end{equation}
    is optimal for guaranteeing stability of $P_m^n$, disregarding both constant factors and factors of $\log n$. That is, up to a factor of $\log n$, \eqref{iff_rate} is both a sufficient and necessary condition for 
    $
    \kappa_{2}\left(P_m^n\right) \lesssim 1
    $ 
    with high probability.
\end{corollary}

\subsubsection{Answer for Question \uppercase\expandafter{\romannumeral2}}
We present an extended impossibility theorem in the context of random sampling, which serves as a generalization of the previous work in \cite{platte2011impossibility, adcock2019optimal}. In this case, the probability measures with respect to the \emph{modified Jacobi} weight functions mentioned above are taken into consideration: 
\begin{equation}\label{eq:mjacobi}
\mu(x) = g(x)(1-x)^\alpha(1+x)^\beta, \quad \text{with} \;\; \int_{-1}^{1}\mu(x){\rm d}x = 1, \;\; x \in X,
\end{equation}
where $\alpha, \beta > -1$ and $g \in L^\infty(X)$ satisfies $c_1\leq g(x)\leq c_2$ almost everywhere.

\begin{theorem}
    \label{impossibility_3}
    Consider a set of random samples $\{x_i\}_{i=1}^{n}\subset X:=[-1,1]$, where $x_i$, $i=1,\dotsc,n$, are i.i.d. with respect to a modified Jacobi weight function characterized by parameters $\alpha,\beta>-1$, and $c_1\leq g(x)\leq c_2$ almost everywhere (see \eqref{eq:mjacobi}). 
    Let $E \subseteq \mathbb{C}$ is a compact set with $X$ contained in its interior. Assume that $\left\{\phi_n\right\}_{n=1}^{\infty}$ is an approximation procedure based on $\{x_i\}_{i=1}^{n}$ such that, for some $M < \infty$, $\sigma>1$, and $\tau > 0$,
    \begin{equation}
        \label{undefine}
        \left\|f-\phi_n(f)\right\|_{\infty} \leq M \sigma^{-n^\tau}\|f\|_E, \quad 1 \leq n < \infty,
    \end{equation}
    for all $f \in B(E)$, with a probability $P_0$. 
    If $\gamma := \max\{\alpha, \beta\} > -1/2$, and 
    $$
    \tau > \frac{1}{2(1+\gamma)},
    $$
    then for any $s \in \mathbb{R}_{+}$ satisfying $s<2(1+\gamma)\tau-1$ and for all large $n$, the following holds with a probability of at least $\left(1-\frac{2e^2\bar{c}}{n^s}\right)P_0$,
    \begin{equation}
        \label{impos_result}
        \kappa_{\infty}\left(\phi_n\right) \geq \eta \,\nu^{n^r}, \quad r = \frac{2(1+\gamma)\tau-1-s}{1+2\gamma}.
    \end{equation}
    Here, $\bar{c} = \frac{c_2 2^{\min\{\alpha, \beta\}}}{1+\gamma}$, and $\eta>0$, $\nu>1$ are constants depending on $\gamma$, while $\nu$ is additionally influenced by the parameter $\sigma$ and the size of $E$.
\end{theorem}

\begin{remark}
Theorem \ref{impossibility_3} outlines the trade-off (in probability) between stability and accuracy. Indeed, under the assumption $\gamma = \max\{\alpha, \beta\} > -1/2$, if an approximation method is stable with high probability, then the convergence rate for approximating analytic functions cannot exceed $\mathcal{O}(\sigma^{-n^{\frac{1}{2(1+\gamma)}}})$ for some $\sigma>1$.

\end{remark}

\begin{remark}
    \label{optimal_convergence}
    Here, we demonstrate that $P_m^n$
  is an approximation method capable of achieving the optimal convergence rate as indicated in Theorem \ref{impossibility_3}. Numerous studies focus on enhancing the reconstruction accuracy of $P_m^n$
  to attain nearly optimal results, i.e., comparable to the best $L^\infty(X)$
  norm approximation error defined by
    \begin{equation*}
	\label{error_inf}
	e_{m,\infty}(f) = \inf_{v \in \mathbb{P}_{m}}\|v-f\|_\infty.
    \end{equation*}
In this context, it is essential to ensure that $f \in L^\infty(X)$; otherwise, $e_{m,\infty}(f)$
  becomes infinite.
    Paper \cite{adcock2019optimal} examines the deterministic equidistributed sampling according to the modified Jacobi weight functions with the sampling rate $n \asymp m^{2(1+\gamma)}$, and establishes that $\|f-P_m^n(f)\|_\infty \leq \mathcal{O}(\sqrt{n}) \, e_{m,\infty}(f)$. 
    We proceed with random sampling regarding the Jacobi weight functions. Given $f \in L^\infty(X)$, following the argumentation in \cite{adcock2019optimal}, it can be asserted that 
    $\|f-P_m^n(f)\|_\infty \leq \mathcal{O}(\sqrt{n}) \, e_{m,\infty}(f)$ with high probability, provided $n \asymp m^{2(1+\gamma)}$ without considering a logarithmic factor of $n$. Hence, for recovering analytic functions $f \in B(E_\rho)$ where $E_\rho \subset \mathbb{C}$ denotes the Bernstein ellipse with parameter $\rho>1$, the well-established bound $e_{m,\infty}(f) \leq \frac{2}{\rho-1}\|f\|_{E_\rho}\rho^{-m}$ (see \cite[Theorem 8.2]{trefethen2019approximation}) suggests that the best possible convergence rate $\mathcal{O}(\sigma^{-n^{\frac{1}{2(1+\gamma)}}})$ can be attained for $P_m^n$ with high probability, under the setting $n \asymp m^{2(1+\gamma)}$.
\end{remark}

\subsection{Organization} 
The paper is organized as follows. Section \ref{2} introduces the notations and lemmas that will be utilized throughout the paper, as well as the primary probability measures considered in this paper. In Section \ref{3}, we conduct an analysis of the maximal behaviour of polynomials that are bounded on $n$ random sampling points. This analysis serves as a tool to address Questions \uppercase\expandafter{\romannumeral1} and \uppercase\expandafter{\romannumeral2}. Section \ref{4} presents the proof of Theorem \ref{opt_rates} and Corollary \ref{opt_sam_rat}, providing a positive response to Question \uppercase\expandafter{\romannumeral1}. Section \ref{5} offers the proof of Theorem \ref{impossibility_3} in relation to Question \uppercase\expandafter{\romannumeral2}. Lastly, in Section \ref{6}, we conduct numerical experiments to validate our theoretical findings.

\section{Preliminaries}
\label{2}
In this section, we present certain notations, lemmas, and the \emph{Jacobi} weight functions that are instrumental to our paper.

\subsection{Notations}
In this paper, we consider functions defined on compact intervals, which can be normalized without loss of generality to $X = [-1, 1]$. Unless otherwise specified, $\{x_i\}_{i=1}^{n} \subset X$ are i.i.d. random sampling points drawn from an absolutely continuous probability measure $\rho_X$ on $X$ with respect to the Lebesgue measure. We evaluate a function $f$ at these points. 
Additionally, we define $\mathbb{P}_{m}$ as the space of univariate algebraic polynomials of degree at most $m$. Throughout the paper, the notation $e$ stands for the base of the natural logarithm.

Let's recapitulate the various norms used in this paper. Firstly, we have the $L^\infty(X)$ norm denoted as $\|f\|_{\infty}:= \sup_{x \in X}{\abs{f(x)}}$. 
Correspondingly, we have the discrete supremum semi-norm, also referred to as the empirical $L^\infty(X)$ semi-norm, denoted as $\|f\|_{n,\infty}:=\max_{i \in \{1,\dotsc,n\}}{\abs{f(x_i)}}$. 
Next, we have the $L^2(X,\rho_X)$ norm, denoted as $\|f\|:= \left( \int_X {|f(x)|^2} {\rm d}\rho_X(x) \right)^{1/2}$, associated with the inner product, denoted as $\left<f,g\right> := \int_X {f(x)\overline{g(x)}} {\rm d}\rho_X(x)$. Correspondingly, there is the empirical $L^2(X,\rho_X)$ semi-norm, denoted as $\|f\|_{n,2}:= \left(\frac{1}{n} \sum_{i=1}^n{|f(x_i)|^2}\right)^{1/2}$, associated with the empirical semi-inner product, denoted as $\left<f,g\right>_n := \frac{1}{n} \sum_{i=1}^n{f(x_i)\overline{g(x_i)}}$. Furthermore, we write $\|\mathbf{u}\|_2$ for the Euclidean norm of a vector $\mathbf{u}$, and $|||\mathbf{M}|||:= \max_{\mathbf{u} \neq 0} \frac{\|\mathbf{Mu}\|_2}{\|\mathbf{u}\|_2}$ for the spectral norm of a matrix $\mathbf{M}$.

We adopt the asymptotic notation $a_n \asymp b_n$ as $n \rightarrow \infty$ if there exist positive constants $c_1$ and $c_2$ such that $0 \leq c_1 b_n \leq a_n \leq c_2 b_n$ holds for all large $n$. For brevity, we will omit the notation $n \rightarrow \infty$. At times, we use the notation $A \lesssim B$ to denote that there exists a universal constant $c>0$ such that $A \leq cB$. We say that $a_n$ converges to zero exponentially fast if $\abs{a_n}=\mathcal{O}\left(\sigma^{-n^\tau}\right)$ for some constants $\sigma>1$ and $\tau>0$. 
Specifically, when $\tau = 1/2$, we refer to it as root-exponential convergence.

\subsection{Random sampling with respect to the modified Jacobi weight functions}
\label{Jacobi}

We consider sampling based on the Jacobi weight functions, which are defined as
\begin{equation}
    \label{Jacobi_weight}
    w(x) = c(1-x)^\alpha(1+x)^\beta, \quad \text{with} \;\; c = \left(\int_{-1}^{1}(1-x)^\alpha(1+x)^\beta {\rm d}x\right)^{-1},
    \;\; x \in X=[-1,1],
\end{equation}
where $\alpha, \beta > -1$. It is well-known that choosing $\alpha = 0$ and $\alpha = -\frac{1}{2}$ correspond to the uniform and Chebyshev weight functions, respectively. The random sampling points $\{x_i\}_{i=1}^{n} \subset X$ are i.i.d with respect to the probability measure $\rho_X$ that satisfies 
${\rm d}\rho_X = w(x){\rm d}x$.

We would like to mention that the \emph{Jacobi} weight functions are specific instances of the the following \emph{modified Jacobi} weight functions
\begin{equation*}
	\label{modified_Jacobi}
	\mu(x) = g(x)(1-x)^\alpha(1+x)^\beta,
    \quad x \in X,
\end{equation*}
where $\int_{-1}^{1}\mu(x){\rm d}x = 1$, $\alpha, \beta > -1$, and $g \in L^\infty(X)$ satisfies the condition $c_1\leq g(x)\leq c_2$ almost everywhere for some constants $c_1,c_2$.


\subsection{Relation between $\kappa_{2}\left(P_m^n\right)$ and  $D(n,m)$}

To comprehend the essence of $\kappa_{2}\left(P_m^n\right)$, we introduce random variables $D(n,m)$ and $B(n,m)$ as follows:
\begin{equation}
    \label{DB}
    D(n,m) := \sup_{\substack{p \in \mathbb{P}_m \\ \|p\|_{n, 2} \neq 0}}{\frac{\|p\|}{\|p\|_{n,2}}}, \quad \quad
    B(n,m) := \sup_{\substack{p \in \mathbb{P}_m \\ \|p\|_{n, \infty} \neq 0}}{\frac{\|p\|_\infty}{\|p\|_{n,\infty}}}, \quad \text{with} \;\; 1\leq m<n.
\end{equation}
The norms $\|p\|_{n,2}$ and $\|p\|_{n,\infty}$  depend on the random sampling points $\{x_i\}_{i=1}^n$ that are i.i.d according to a absolutely continuous probability measure $\rho_X$ on $X$ with respect to the Lebesgue measure. As a result, both $D(n,m)$ and $B(n,m)$ are affected by the sampling points $\{x_i\}_{i=1}^n$. However, in order to maintain simplicity in notation, we have chosen not to explicitly represent this dependence in our symbols. Let $\Omega$ be the event where there exist at least $m+1$ of the sampling points $\{x_i\}_{i=1}^n$ that are distinct. Evidently, we have $P(\Omega)=1$. And both terms in \eqref{DB} are valid on $\Omega$ and infinite otherwise. Thus, all the following results involving the properties of $D(n,m)$ or $B(n,m)$ are demonstrated disregarding the set $\Omega^{C}$ of measure zero. 

The connection between the two of them can be readily obtained from Lemma \ref{norm_relation}, which employs the quantity $K(m)$, as defined in \eqref{quantity}, to regulate the $L^{\infty}(X)$ norm using the $L^2(X, \rho_X)$ norm. We first introduce a classical result concerning \emph{Jacobi}  polynomials, which will facilitate our proof.
\begin{theorem}[{\cite[\S 2.3, Lemma 3.4, Theorem 3.11]{osilenker1999fourier}}]
    \label{Jacobi_poly}
    Let $X=[-1,1]$ and $\rho_X$ be a probability measure on $X$ such that ${\rm d}\rho_X = w(x){\rm d}x$, where $w(x)$ denotes a Jacobi weight function characterized by parameters $\alpha,\beta>-1$ as in \eqref{Jacobi_weight}. 
    Consider the \emph{Jacobi} polynomials with the parameters $\alpha,\beta$
    \begin{equation}
        \label{jacoPolyDef}
        J_k(\alpha,\beta;x):= \frac{(-1)^k}{k!2^k}(1-x)^{-\alpha}(1+x)^{-\beta}\frac{{\rm d}^k}{{\rm d}x^k}[(1-x)^{\alpha+k}(1+x)^{\beta+k}], \quad k \in \mathbb{Z}_{\geq 0},\; x \in X.
    \end{equation}
    They are orthogonal in the sense of $L^2(X,\rho_X)$.
    And if $\gamma := \max\{\alpha,\beta\} > -1/2$, then there exists a constant $C$ depending only on $\alpha, \beta$ such that 
    $$
    \|J_k\|_\infty \leq C\,k^{\gamma+1/2}\,\|J_k\|, \quad \text{for all} \;\; k \in \mathbb{Z}_{\geq 0}.
    $$
\end{theorem}

\begin{lemma}
	\label{norm_relation}
    Let $X$ be a domain of $\mathbb{R}^d$ and $\rho_X$ be a probability measure on $X$. 
	Assume $V_m$ is an $m$-dimensional linear subspace of $L^{\infty}(X) \cap L^2(X,\rho_X)$, equipped with an orthonormal basis
	$\{L_j\}_{j=1}^{m}$ in the sense of $L^2(X,\rho_X)$. Then, 
    \begin{equation}
        \label{ie_1}
        \|v\|_{\infty} \leq \sqrt{K(m)} \cdot \|v\|, \quad \text{ for all} \;\; v \in V_m,
    \end{equation}
	where $K(m):=\sup_{x \in X} \sum_{j=1}^{m}{|L_j(x)|^2}$, as outlined in \eqref{quantity}. In particular, when considering $X=[-1,1]$, $V_{m+1} = \mathbb{P}_m$ and $\rho_X$ such that ${\rm d}\rho_X = w(x){\rm d}x$ where $w(x)$ denotes a Jacobi weight function characterized by parameters $\alpha,\beta>-1$ as in \eqref{Jacobi_weight}, if $\gamma := \max\{\alpha,\beta\} > -1/2$, then for each $m \geq 1$,
    \begin{equation}
        \label{ie_2}
        \|p\|_{\infty} \leq  C\,m^{1+\gamma}\, \|p\|, \quad \text{for all} \;\; p \in \mathbb{P}_m,
    \end{equation}
    where the constant $C$ depends only on $\alpha,\beta$.
\end{lemma}
\begin{proof}
	For any $v=\sum_{j=1}^{m}{w_j L_j} \in V_m$ with $\bm{w}:=\left(w_1,\ldots,w_m\right)^\top$, 
	$$
	\|v\|_{\infty} = \sup_{x \in X} \left| \sum_{j=1}^{m}{w_j L_j(x)} \right| 
	\leq \left(\sup_{x \in X} \sum_{j=1}^{m}{|L_j(x)|^2}\right)^{\frac{1}{2}} \cdot \|\bm{w}\|_2
	= \sqrt{K(m)} \cdot \|v\|.
	$$
Here, we apply the Cauchy–Schwarz inequality.

To prove \eqref{ie_2} by utilizing \eqref{ie_1}, it remains to estimate $K(m+1)$. To this end, we take $L_k = J_{k-1}/\|J_{k-1}\|$ for $k=1,\dotsc,m+1$, where $J_{k-1}$ is defined as in \eqref{jacoPolyDef}, to be the $L^2(X,\rho_X)$-normalized \emph{Jacobi} polynomials with the parameters $\alpha,\beta$. Note that $L_k$ is of degree $k-1$. Hence, by Theorem \ref{Jacobi_poly}, $\{L_k\}_{k=1}^{m+1}$ forms an orthonormal basis of $\mathbb{P}_m$ in the sense of $L^2(X,\rho_X)$. Since $\gamma = \max\{\alpha,\beta\} > -1/2$, Theorem \ref{Jacobi_poly} further implies that there exists a constant $C_1$ depending only on $\alpha, \beta$ such that 
$$
\|L_k\|_\infty \leq C_1\,k^{\gamma+1/2}, \quad 1 \leq k \leq m+1.
$$
Consequently, we have
\begin{equation}
    \label{below_use}
    K(m+1)=\sup_{x \in X} \sum_{k=1}^{m+1}{|L_k(x)|^2} \leq C_1\sum_{k=1}^{m+1}{k^{1+2\gamma}} \leq C_1(m+2)^{2(1+\gamma)} \leq  C\,m^{2(1+\gamma)},
\end{equation}
where $C$ is a constant depending only on $\alpha, \beta$. This in conjugation with \eqref{ie_1} completes the proof of \eqref{ie_2}.
\end{proof}

\begin{lemma}
    \label{re_norm}
    Given random samples $\{x_i\}_{i=1}^{n}\subset X=[-1,1]$ i.i.d. with respect to a Jacobi weight function characterized by parameters $\alpha,\beta>-1$, assume that $\gamma := \max\{\alpha,\beta\} > -1/2$, we have
    $$
    \frac{B(n,m)}{C\,m^{1+\gamma}}
    \leq D(n,m)
    \leq \sqrt{n}B(n,m) \quad \text{almost surely}, 
    $$
    for $1 \leq m < n$ and some constant $C$ depending only on $\alpha,\beta$. Here, the definitions of $B(n,m)$ and $D(n,m)$ are provided in \eqref{DB}.
\end{lemma}
\begin{proof}
    We establish the inequalities within the domain $\Omega$
    ensuring that there are at least $m+1$ distinct sampling points among $\{x_i\}_{i=1}^n$. Lemma \ref{norm_relation} states that for each $m \geq 1$, $\|p\|_\infty \leq  C\,m^{1+\gamma}\|p\|$ for all $p \in \mathbb{P}_m$ and some constant $C$ depending only on $\alpha,\beta$. It combined with the fact that $\|p\|_{n, \infty} \geq \|p\|_{n,2}$ provides that
    $$
    \frac{B(n,m)}{C\,m^{1+\gamma}}
    = \sup_{\substack{p \in \mathbb{P}_m \\ \|p\|_{n, \infty} \neq 0}}{\frac{\|p\|_\infty/(C\,m^{1+\gamma})}{\|p\|_{n,\infty}}}
    \leq \sup_{\substack{p \in \mathbb{P}_m \\ \|p\|_{n, 2} \neq 0}}{\frac{\|p\|}{\|p\|_{n,2}}} = D(n,m), 
    $$
    for $1 \leq m < n$. Additionally, the norm inequalities $\|p\| \leq \|p\|_\infty$ and $\sqrt{n}\|p\|_{n,2} \geq \|p\|_{n,\infty}$ yield that
    $$
    D(n,m) = \sup_{\substack{p \in \mathbb{P}_m \\ \|p\|_{n, 2} \neq 0}}{\frac{\|p\|}{\|p\|_{n,2}}}
    \leq \sup_{\substack{p \in \mathbb{P}_m \\ \|p\|_{n, \infty} \neq 0}}{\frac{\|p\|_\infty}{\|p\|_{n,\infty}/\sqrt{n}}}
    = \sqrt{n}B(n,m).
    $$
\end{proof}

The following lemma claims that $\kappa_{2}\left(P_m^n\right)$ and $D(n,m)$ are actually equal almost surely. This also indicates that $\kappa_{2}\left(P_m^n\right)$ is irrelevant to the function space where the function $h$, in the definition \eqref{cond_2} of $\kappa_{2}\left(P_m^n\right)$, is located.

\begin{lemma}
    \label{eq}
    Given random samples $\{x_i\}_{i=1}^{n}\subset X=[-1,1]$ i.i.d. with respect to a absolutely continuous probability measure $\rho_X$ on $X$ with respect to the Lebesgue measure,
    we have
    $$
    \kappa_{2}\left(P_m^n\right) = D(n,m) = \left(\lambda_{\min}(\mathbf{G})\right)^{-1/2} \quad \text{almost surely},
    $$  
    for $1 \leq m < n$. Here, $\kappa_{2}\left(P_m^n\right)$ and $D(n,m)$ are defined in \eqref{cond_2} and \eqref{DB} respectively, and $\mathbf{G} :=(\left<L_j, L_k\right>_n)_{j,k=1,\dotsc,m+1}$ is a positive semidefinite matrix as in \eqref{system} with $\{L_j\}_{j=1}^{m+1}$ an orthonormal basis of $\mathbb{P}_m$ in the sense of $L^2(X,\rho_X)$.
\end{lemma}
\begin{proof}
     We still assume that $\Omega$ holds, ensuring that there are at least $m+1$ distinct sampling points among $\{x_i\}_{i=1}^n$.
     Assuming that $\{L_j\}_{j=1}^{m+1}$ is an orthonormal basis of $\mathbb{P}_{m}$ in the sense of $L^2(X,\rho_X)$, any $p \in \mathbb{P}_{m}$ can be represented as $p = \sum_{j=1}^{m+1}{w_jL_j}$. It follows that $\|p\|^2=\|\mathbf{w}\|_2^2$ with $\mathbf{w} = (w_1,\dotsc,w_m)^\top$ and that $\|p\|_{n,2}^2=\mathbf{w}^\top\mathbf{G}\mathbf{w}$. It is apparent that $\|p\|_{n, 2} = 0$ is equivalent to $p = 0$. Therefore, $\mathbf{G}$ is nonsingular.
    
    Since the operator $P_m^n$ is a projection onto $\mathbb{P}_m$, we have 
    \begin{equation}
        \label{ge}
        \kappa_{2}\left(P_m^n\right) 
        = \sup _{\substack{h \\ \|h\|_{n, 2} \neq 0}} \frac{\left\|P_m^n(h)\right\|}{\|h\|_{n, 2}}
        \geq \sup_{\substack{p \in \mathbb{P}_m \\ \|p\|_{n, 2} \neq 0}}{\frac{\|P_m^n(p)\|}{\|p\|_{n,2}}}
        =\sup_{\substack{p \in \mathbb{P}_m \\ \|p\|_{n, 2} \neq 0}}{\frac{\|p\|}{\|p\|_{n,2}}}
        = D(n,m).
    \end{equation}
    On the other hand, it can be observed that the supremum of $\frac{\left\|P_m^n(h)\right\|}{\|h\|_{n, 2}}$ where $\|h\|_{n, 2} \neq 0$ cannot be reached at $h$ such that $\|P_m^n(h)\|=0$, i.e., $\|P_m^n(h)\|_{n, 2}=0$. Consequently, 
    $$
    \begin{aligned}
    \kappa_{2}\left(P_m^n\right)
    & = \sup _{\substack{h \\ \|h\|_{n, 2} \neq 0 \\ \|P_m^n(h)\|_{n, 2} \neq 0}} \frac{\left\|P_m^n(h)\right\|}{\|h\|_{n, 2}}
    = \sup _{\substack{h \\ \|h\|_{n, 2} \neq 0 \\ \|P_m^n(h)\|_{n, 2} \neq 0}} \left(\frac{\left\|P_m^n(h)\right\|}{\|h\|_{n, 2}} \frac{\|P_m^n(h)\|_{n, 2}}{\|P_m^n(h)\|_{n, 2}}\right) \\
    & \leq D(n,m)\sup _{\substack{h \\ \|h\|_{n, 2} \neq 0 \\ \|P_m^n(h)\|_{n, 2} \neq 0}} \frac{\left\|P_m^n(h)\right\|_{n, 2}}{\|h\|_{n, 2}} 
    \leq D(n,m),
    \end{aligned}
    $$
    where the final inequality follows  from the fact that $P_m^n$ is the orthogonal projection with respect to the empirical semi-inner product $\left<\cdot,\cdot\right>_n$. Combined with \eqref{ge}, $\kappa_{2}\left(P_m^n\right) \overset{a.s.}{=} D(n,m)$ is obtained.
    
   Furthermore, 
   \begin{equation*}
   \begin{aligned}
        D(n,m) & = \sup_{\substack{p \in \mathbb{P}_m \\ 
        \|p\|_{n, 2} \neq 0}}{\frac{\|p\|}{\|p\|_{n,2}}}
        \overset{(a)}{=} \sup_{\substack{p \in \mathbb{P}_m \\ p \neq 0}}{\frac{\|p\|}{\|p\|_{n,2}}} \\
        & = \left(\sup_{\substack{\mathbf{w} \in \mathbb{R}^{m} \\ \mathbf{w} \neq 0}}{\frac{\mathbf{w}^\top\mathbf{w}}{\mathbf{w}^\top\mathbf{G}\mathbf{w}}}\right)^{1/2}
        = \left(\sup_{\substack{\mathbf{w} \in \mathbb{R}^{m} \\ \|\mathbf{w}\|_2 = 1}}{\frac{1}{\mathbf{w}^\top\mathbf{G}\mathbf{w}}}\right)^{1/2} \\
        & \overset{(b)}{=} \left(\max_{\substack{\mathbf{w} \in \mathbb{R}^{m} \\ \|\mathbf{w}\|_2 = 1}}{\frac{1}{\mathbf{w}^\top\mathbf{G}\mathbf{w}}}\right)^{1/2}
        = \left(\min_{\substack{\mathbf{w} \in \mathbb{R}^{m} \\ \|\mathbf{w}\|_2 = 1}}{\mathbf{w}^\top\mathbf{G}\mathbf{w}}\right)^{-1/2}
        = \left(\lambda_{\min}(\mathbf{G})\right)^{-1/2}, 
    \end{aligned}
    \end{equation*}
    where the reason for ($a$) is that $\|p\|_{n, 2} \neq 0$ is equivalent to $p \neq 0$ on $\Omega$, and ($b$) is derived from the fact that a continuous function on a compact set attains its maximum value. Note that $P(\Omega)=1$. This completes the proof.
\end{proof}

\subsection{Maximal behavior of polynomials bounded on deterministic samples}

Here, our focus lies on $B(n,m)$, as defined in \eqref{DB}, which signifies the highest attainable value of polynomials within the range $[-1,1]$ with specific constraints. Since it is a fundamental quantity, previous research has already reported explicit results for deterministic samples.
 The proof of Theorem \ref{impossibility} incorporates the following lemma from Coppersmith and Rivlin \cite{coppersmith1992growth} regarding equispaced points in $[-1,1]$.

\begin{lemma}[\cite{coppersmith1992growth}]
    \label{polynomials}
  Given samples $\{x_i\}_{i=1}^{n}$ as the equispaced grids of $n$ points on $X=[-1,1]$, there exist universal constants $\nu_1,\nu_2 >1$  as well as $n_0 \geq 1$, such that
    $$
    \nu_1^{\frac{m^2}{n}} \leq B(n,m) \leq \nu_2^{\frac{m^2}{n}}
    $$
holds for all $n \geq n_0$ and $1 \leq m < n$.
\end{lemma}

Subsequently, in Lemma \ref{polynomials——deter}, the paper \cite{adcock2019optimal} extends the findings of Lemma \ref{polynomials} to encompass deterministic equidistributed sampling with regard to the \emph{modified Jacobi} weight functions.
 This extension ultimately leads to the derivation of Theorem \ref{impossibility_4}.

\begin{lemma}[\cite{adcock2019optimal}]
	\label{polynomials——deter}
	Given samples $\{x_i\}_{i=1}^{n} \subset X=[-1,1]$, where $x_i$, $i=1,\dotsc,n$, are deterministically equidistributed with respect to a modified Jacobi weight function with parameters $\alpha,\beta>-1$ and $c_1\leq g(x)\leq c_2$ almost everywhere, if $\gamma := \max\{\alpha, \beta\} > -1/2$, then there exist constants $\eta>0$, $\nu>1$ depending on $\alpha$ and $\beta$, such that
	$$
	B(n,m) \geq \eta \,\nu^{\lambda}, \quad \lambda = \left(\frac{m^{2(1+\gamma)}}{n}\right)^{\frac{1}{1+2\gamma}},
	$$
	for all $1 \leq m < n$.
\end{lemma}


\section{The maximal behavior of polynomials bounded on random sampling points}
\label{3}

In this section, our focus is on investigating the behavior of $B(n,m)$ defined as follows:
$$
B(n,m) = \sup_{\substack{p \in \mathbb{P}_m \ \|p\|_{n, \infty} \neq 0}}{\frac{\|p\|_\infty}{\|p\|_{n,\infty}}}.
$$
This analysis is conducted considering i.i.d random sampling points $\{x_i\}_{i=1}^{n} \subset X = [-1,1]$. By leveraging the insights from Lemma \ref{polynomials} and Lemma \ref{polynomials——deter}, we can understand the influence of random sampling rates on $B(n,m)$, which allows us to further determine the behavior of $\kappa_{2}(P_m^n)=D(n,m)$. Ultimately, this investigation contributes to addressing Question \uppercase\expandafter{\romannumeral1} effectively.

\subsection{The lower bound of $B(n,m)$}

The following is the foundational result in this section that gives exponential lower bound of $B(n,m)$. We present its proof in Section 3.2.

\begin{lemma}
    \label{polynomials_random}
     Assume that $\{x_i\}_{i=1}^{n}\subset X=[-1,1]$ are i.i.d. random samples with respect to a modified Jacobi weight function with parameters $\alpha,\beta>-1$ and $c_1\leq g(x)\leq c_2$ almost everywhere. If $\gamma := \max\{\alpha, \beta\} > -1/2$, then there exist constants $\eta>0$, $\nu>1$ depending on $\gamma$, such that for any $s \in \mathbb{R}_{+}$ and all $1 \leq m < n$ with $n>(2e^2\bar{c})^{1/s}$ where $\bar{c} = \frac{c_2 2^{\min\{\alpha, \beta\}}}{1+\gamma}$, we have
     \begin{equation*}
     \label{maximal_behavior}
         B(n,m) \geq \eta \,\nu^\lambda, \quad \lambda = \left(\frac{m^{2(1+\gamma)}}{n^{1+s}}\right)^{\frac{1}{1+2\gamma}},
     \end{equation*}
     with probability at least $1-\frac{2e^2\bar{c}}{n^s}$.
\end{lemma}

By combining Lemma \ref{re_norm} and Lemma \ref{polynomials_random}, we can derive the following results for $D(n,m)$ directly. It should be pointed out that these results are applicable to sampling based on the \emph{Jacobi} weight functions
\begin{equation}
    \label{Jacobi_ee}
    w(x) = c(1-x)^\alpha(1+x)^\beta, \quad \text{with} \quad c = \left(\int_{-1}^{1}(1-x)^\alpha(1+x)^\beta {\rm d}x\right)^{-1},\;\;\alpha, \beta > -1
\end{equation}
rather than the modified ones.
\begin{corollary}
    \label{L2_behavior}
    Assume that $\{x_i\}_{i=1}^{n}\subset X=[-1,1]$ are i.i.d. random samples with respect to a Jacobi weight function with parameters $\alpha,\beta>-1$. If $\gamma := \max\{\alpha, \beta\} > -1/2$, then there exist constants $\eta>0$, $\nu>1$ depending on $\alpha,\beta$, such that for any $s >0$ and all $1 \leq m < n$ with $n>(2e^2\bar{c})^{1/s}$ where $\bar{c} = \frac{c \cdot 2^{\min\{\alpha, \beta\}}}{1+\gamma}$ with the constant $c$ from \eqref{Jacobi_ee}, we have
    $$
    D(n,m) 
    \geq \frac{\eta}{m^{1+\gamma}} \nu^\lambda, \quad \lambda = \left(\frac{m^{2(1+\gamma)}}{n^{1+s}}\right)^{\frac{1}{1+2\gamma}},
    $$
    with probability at least $1-\frac{2e^2\bar{c}}{n^s}$.
\end{corollary}

By combining Corollary \ref{L2_behavior} with Theorem \ref{cohen}, we obtain the following result.

\begin{corollary}
    \label{bound_iff_cond}
    Assume that $\{x_i\}_{i=1}^{n}\subset X=[-1,1]$ are i.i.d. random samples with respect to a Jacobi weight function with parameters $\alpha,\beta>-1$. If $\gamma = \max\{\alpha, \beta\} > -1/2$, then the necessary and sufficient condition for $D(n,m) \lesssim 1$ with high probability is 
    $$
    n \asymp m^{2(1+\gamma)}
    $$
    without considering the logarithmic factor of $n$.
\end{corollary}

\begin{proof}[Proof of Corollary \ref{bound_iff_cond}]
    Necessity is already revealed in Corollary \ref{L2_behavior}. Indeed, let $n \asymp m^{1/\tau}$, i.e., there exist positive constants $c_1$ and $c_2$ such that $c_1 n^\tau \leq m \leq c_2 n^\tau$ holds for all large $n$. If $\tau > \frac{1}{2(1+\gamma)}$, there exists a constant $s \in \mathbb{R}_{+}$ satisfying $2(1+\gamma)\tau-1-s>0$. Then it directly follows from Corollary \ref{L2_behavior} that, there exist constants $\eta>0$, $\nu>1$ depending on $\alpha,\beta$, such that
    $$
    D(n,m) 
    \geq \frac{\eta}{c_2^{1+\gamma}n^{(1+\gamma)\tau}} \nu^{\lambda_1}, \quad \lambda_1 = \left(c_1^{2(1+\gamma)}n^{2(1+\gamma)\tau-1-s}\right)^{\frac{1}{1+2\gamma}},
    $$
    for all large $n$, with probability at least $1-\frac{2e^2\bar{c}}{n^s}$, where $\bar{c} = \frac{c \cdot 2^{\min\{\alpha, \beta\}}}{1+\gamma}$ with the constant $c$ from \eqref{Jacobi_ee}. Therefore, in order to have $D(n,m) \lesssim 1$ with high probability, it is required that $\tau \leq \frac{1}{2(1+\gamma)}$. This confirms the necessity.
    
    We next show that this condition, up to an additional logarithmic factor of $n$, is sufficient. According to Theorem \ref{cohen}, it is enough to prove 
    \begin{equation}
        \label{cohn_cond}
        K(m+1) \leq C\,m^{2(1+\gamma)},
    \end{equation}
    for some constant $C$ independent of $m$, when $\{L_j\}_{j=1}^{m+1}$, as defined in \eqref{quantity}, is an orthonormal basis of $\mathbb{P}_m$ on $[-1,1]$ with respect to the Jacobi weight function. 
    In fact, if condition \eqref{cohn_cond} holds, then combining \eqref{cohn_cond} with the condition $n \asymp m^{2(1+\gamma)}$ disregarding both the constant factor and the factor of $\log n$, we can deduce by Theorem \ref{cohen} that for any function $f:X\rightarrow \mathbb{R}$, 
    $
    \|P_m^n(f)\| \leq \sqrt{6}\; \|f\|_{n,2}
    $
    with high probability. 
    It follows that the condition number for $P_m^n$ as in \eqref{cond_2} satisfies 
    $
    \kappa_2(P_m^n) \lesssim 1
    $
    with high probability. Note that $ \kappa_{2}\left(P_m^n\right) = D(n,m)$ almost surely, as shown in Lemma \ref{eq}. Therefore, $D(n,m) \lesssim 1$ also holds with high probability.
    
    The establishment of \eqref{cohn_cond} has been included in the demonstration of Lemma \ref{norm_relation} (see \eqref{below_use}). Thus, we complete the proof.
\end{proof}

\subsection{The proof of Lemma \ref{polynomials_random}}


The relationship between the behavior of polynomials described in Lemma \ref{polynomials_random} and the distribution of random sampling points is intricate. It requires an investigation into the theory of order statistics. Hence, before we delve into the proof of Lemma \ref{polynomials_random}, we introduce Lemma \ref{prob_estimate} as a preliminary step. This lemma is of  importance in establishing Lemma \ref{polynomials_random}.


\begin{lemma}
    \label{prob_estimate} 
     Assume that $\{x_i\}_{i=1}^{n}\subset X=[-1,1]$ are i.i.d. random samples with respect to a modified Jacobi weight function with parameters $\alpha,\beta>-1$ and $c_1\leq g(x)\leq c_2$ almost everywhere.
    Let $x_{(1)} \leq x_{(2)} \leq \cdots \leq x_{(n)}$ be the corresponding order statistics. Then, for any $s \in \mathbb{R}_{+}$ and $n>(2e^2\bar{c}_1)^{1/s}$ where $\bar{c}_1 = \frac{c_2 2^\alpha}{1+\beta}$, we have
    $$
    P\left(x_{(1)} \geq \left(\frac{1}{n^{1+s}}\right)^{\frac{1}{1+\beta}}-1, x_{(2)} \geq \left(\frac{2}{n^{1+s}}\right)^{\frac{1}{1+\beta}}-1, \ldots, x_{(n)} \geq \left(\frac{n}{n^{1+s}}\right)^{\frac{1}{1+\beta}}-1 \right) \geq 1-\frac{2e^2\bar{c}_1}{n^s}.
    $$
    
\end{lemma}

\begin{proof}
    We have
    \begin{equation}\label{eq:budeng}
    \begin{aligned}        
        & P\left(x_{(1)} \geq \left(\frac{1}{n^{1+s}}\right)^{\frac{1}{1+\beta}}-1, x_{(2)} \geq \left(\frac{2}{n^{1+s}}\right)^{\frac{1}{1+\beta}}-1, \ldots, x_{(n)} \geq \left(\frac{n}{n^{1+s}}\right)^{\frac{1}{1+\beta}}-1 \right)  \\
        & = 1-P\left(\exists \; k \in \{1,\ldots,n\} \;\; \text{s.t.} \;\; 
        x_{(k)} < \left(\frac{k}{n^{1+s}}\right)^{\frac{1}{1+\beta}}-1 \right)  \\
        & = 1-P\left(\underset{k=1}{\overset{n}{\cup}}\left\{x_{(k)} < \left(\frac{k}{n^{1+s}}\right)^{\frac{1}{1+\beta}}-1\right\}
         \right) \\
        & \geq 1 - \sum_{k=1}^{n}{P\left({x_{(k)} \leq \left(\frac{k}{n^{1+s}}\right)^{\frac{1}{1+\beta}}}-1\right)}.
    \end{aligned}
    \end{equation}
    
    Now we focus on $P\left({x_{(k)} \leq \left(\frac{k}{n^{1+s}}\right)^{\frac{1}{1+\beta}}}-1\right)$, $k = 1,\dotsc,n$. 
    It's important to note that the cumulative distribution function of $x_{(k)}$ is given by 
    $$
    F_{x_{(k)}}(u) = \sum_{j=k}^{n}{\tbinom{n}{j}(F(u))^j(1-F(u))^{n-j}}, \quad u \in [-1,1],
    $$
    for all $k=1, \dotsc, n$.
    Here, 
    $$
    F(u) = \int_{-1}^{u}g(x)(1-x)^\alpha(1+x)^\beta {\rm d}x,
    $$
    represents the cumulative distribution function of $x_j$; for further details, please refer to \cite[p. 9]{david2004order}.
    Then it follows that
    \begin{equation}\label{single_prob}
    \begin{aligned}
        P\left(x_{(k)} \leq \left(\frac{k}{n^{1+s}}\right)^{\frac{1}{1+\beta}}-1\right)&=  F_{x_{(k)}}\left(\left(\frac{k}{n^{1+s}}\right)^{\frac{1}{1+\beta}}-1\right)\\
        & = \sum_{j=k}^{n}{\frac{n!}{j!(n-j)!}\left(F\left(\left(\frac{k}{n^{1+s}}\right)^{\frac{1}{1+\beta}}-1\right)\right)^j\left(1-F\left(\left(\frac{k}{n^{1+s}}\right)^{\frac{1}{1+\beta}}-1\right)\right)^{n-j}} \\
        & \leq \sum_{j=k}^{n}{\frac{n!}{j!(n-j)!}\left(F\left(\left(\frac{k}{n^{1+s}}\right)^{\frac{1}{1+\beta}}-1\right)\right)^j}.
    \end{aligned}
    \end{equation}
    Combining \eqref{single_prob} and
    \begin{align*}
        F\left(\left(\frac{k}{n^{1+s}}\right)^{\frac{1}{1+\beta}}-1\right) 
        & = \int_{-1}^{\left(\frac{k}{n^{1+s}}\right)^{1/(1+\beta)}-1}g(x)(1-x)^\alpha(1+x)^\beta {\rm d}x \; \\
        & \leq \; c_2 2^\alpha \int_{-1}^{\left(\frac{k}{n^{1+s}}\right)^{1/(1+\beta)}-1}(1+x)^\beta {\rm d}x 
        = \frac{\bar{c}_1k}{n^{1+s}},
    \end{align*}
    we obtain that
   \begin{equation}\label{eq:pxk}
    \begin{aligned}
        P\left(x_{(k)} \leq \left(\frac{k}{n^{1+s}}\right)^{\frac{1}{1+\beta}}-1\right)
        & \leq \sum_{j=k}^{n}{\frac{n!}{j!(n-j)!}\left(\frac{\bar{c}_1k}{n^{1+s}}\right)^j} 
        = \sum_{j=k}^{n}{\frac{n(n-1)\cdots(n-j+1)}{n^j}\cdot\frac{1}{j!}\left(\frac{\bar{c}_1k}{n^s}\right)^j} \\
        & \leq \sum_{j=k}^{n}{\frac{1}{j!}\left(\frac{\bar{c}_1k}{n^s}\right)^j} 
        \leq \sum_{j=k}^{\infty}{\frac{1}{j!}\left(\frac{\bar{c}_1k}{n^s}\right)^j} \\
        & \overset{(a)}{\leq} \frac{1}{k!}\cdot \left(\frac{\bar{c}_1k}{n^s}\right)^k\cdot e^\frac{\bar{c}_1k}{n^s} 
        = \frac{1}{k!}\cdot \left(\frac{k}{e}\right)^k\cdot\left(\frac{\bar{c}_1}{n^s}\right)^k\cdot e^{k(\bar{c}_1/n^s+1)} \\
        & \overset{(b)}{\leq} \left(\frac{\bar{c}_1}{n^s}\right)^ke^{k(\bar{c}_1/n^s+1)} 
        \overset{(c)}{\leq} \left(\frac{e^{2}\bar{c}_1}{n^s}\right)^k.
    \end{aligned}
   \end{equation}  
   Here, in the inequality (a), we utilize the remainder term estimation of order $k-1$ in the Taylor expansion $e^{\left(\frac{\bar{c}_1k}{n^s}\right)} = \sum_{j=k}^{\infty}{\frac{1}{j!}\left(\frac{\bar{c}_1k}{n^s}\right)^j}$ of the function $e^x$ at $x=\frac{\bar{c}_1k}{n^s}$. And the inequality (b) is deduced from the non-asymptotic version of bounds based on Stirling's formula (see \cite{robbins1955remark}): 
    \begin{equation}
        \label{Stirling}
        \sqrt{2\pi k}\left(\frac{k}{e}\right)^ke^{\frac{1}{12k+1}}   <   k!    < \sqrt{2\pi k}\left(\frac{k}{e}\right)^ke^{\frac{1}{12k}}, \quad k \geq 1.
    \end{equation}
    The inequality (c) is due to $n^s>2e^2\bar{c}_1$.
    
    Combining \eqref{eq:pxk} and \eqref{eq:budeng}, we obtain that 
    \begin{equation*}
    \begin{aligned}
        & P\left(x_{(1)} \geq \left(\frac{1}{n^{1+s}}\right)^{\frac{1}{1+\beta}}-1, x_{(2)} \geq \left(\frac{2}{n^{1+s}}\right)^{\frac{1}{1+\beta}}-1, \ldots, x_{(n)} \geq \left(\frac{n}{n^{1+s}}\right)^{\frac{1}{1+\beta}}-1 \right) \\
        & \geq 1 - \sum_{k=1}^{n}{\left(\frac{e^{2}\bar{c}_1}{n^s}\right)^k}
        = 1 - \frac{\frac{e^2\bar{c}_1}{n^s}\left(1-\left(\frac{e^{2}\bar{c}_1}{n^s}\right)^n\right)}{1-\frac{e^2\bar{c}_1}{n^s}} \geq 1 - \frac{2e^{2}\bar{c}_1}{n^s},
    \end{aligned}
    \end{equation*}
    where in the last inequality we use $n^s>2e^2\bar{c}_1$ again.
\end{proof}

\begin{remark}
    Based on the symmetry of the sampling interval $[-1,1]$, we can employ a similar approach to establish that for any $s>0$ and $n>(2e^2\bar{c}_2)^{1/s}$ where 
    $\bar{c}_2 = \frac{c_2 2^\beta}{1+\alpha}$, the following inequality holds:
    $$
    P\left(x_{(n)} \leq 1-\left(\frac{1}{n^{1+s}}\right)^{\frac{1}{1+\beta}}, x_{(n-1)} \leq 1-\left(\frac{2}{n^{1+s}}\right)^{\frac{1}{1+\beta}}, \ldots, x_{(1)} \leq 1-\left(\frac{n}{n^{1+s}}\right)^{\frac{1}{1+\beta}} \right) \geq 1-\frac{2e^2\bar{c}_2}{n^s}.
    $$
\end{remark}

We are now prepared to present the proof of Lemma \ref{polynomials_random}, which follows the proof principles established by Coppersmith and Rivlin 
\cite{coppersmith1992growth}
 and further developed in the work of  Adcock, Platte, and Shadrin  \cite{adcock2019optimal}.

\begin{proof}[Proof of Lemma \ref{polynomials_random}]
    Considering the symmetry of the sampling interval $[-1,1]$, we can assume, without loss of generality, that $\beta \geq \alpha$. We will now proceed to divide the proof into two cases.
    
    \textbf{Case \romannumeral1}. We first consider the case where 
    $\left(\frac{m^{2(1+\beta)}}{n^{1+s}}\right)^{\frac{1}{1+2\beta}} > (2\pi^2)^{\frac{1+\beta}{1+2\beta}} =: \hat{c}$. Let
    $$
    K := \left\lfloor \frac{1}{(2\pi^2)^{\frac{1+\beta}{1+2\beta}}}\left(\frac{m^{2(1+\beta)}}{n^{1+s}}\right)^{\frac{1}{1+2\beta}}\right \rfloor
    $$
    and observe that $1 \leq K \leq m-1$, as $m \leq n-1$. Assuming that the inequality
    \begin{equation}
        \label{compare_i}
        x_{(j)} \geq \left(\frac{j}{n^{1+s}}\right)^{\frac{1}{1+\beta}}-1, \quad j = 1,\dotsc,K,
    \end{equation}
    has been established, we can further deduce the inequality
    \begin{equation}
        \label{compare_y}
        x_{(j)} \geq y_j, \quad j = 1,\dotsc,K,
    \end{equation}
    where 
    $$
    y_j = -\cos{\frac{\pi(2j+1)}{2m}}, \quad j=0,\dotsc,m-1
    $$
    are the zeros of the Chebyshev polynomial
    $
    q(x) := \cos{(m\arccos{(x)})}
    $
   which has a degree of $m$.
   Now, let's provide a brief explanation of how \eqref{compare_y} is obtained.
    By combining the inequalities
    \begin{equation}
    \label{yj_ineq}
        y_j = -\cos{\frac{\pi(2j+1)}{2m}} \leq \frac{\pi^2(2j+1)^2}{8m^2} - 1 \leq \frac{2\pi^2j^2}{m^2} - 1, \quad j \geq 1,
    \end{equation}
   and \eqref{compare_i}, we can derive that $x_{(j)} \geq y_j$ as long as 
    $$
    \frac{2\pi^2j^2}{m^2} \leq \left(\frac{j}{n^{1+s}}\right)^{\frac{1}{1+\beta}},
    $$
    which can be simplified to
    $$
    j \leq \left\lfloor \frac{1}{(2\pi^2)^{\frac{1+\beta}{1+2\beta}}}\left(\frac{m^{2(1+\beta)}}{n^{1+s}}\right)^{\frac{1}{1+2\beta}}\right\rfloor=K.
    $$
    Set
    $$
    p(x) := \frac{1}{2}q(x)\prod_{j=0}^{K}\frac{x-x_{(j)}}{x-y_j}\in \mathbb{P}_m,
    $$
    where $x_{(0)} := -1$.
    We claim that 
    \begin{equation}\label{eq:pxj}
    p(x_{(k)}) \leq 1, \quad k = 1,\dotsc,n.
    \end{equation}
    In fact, for $k = 1,\dotsc,K$, we have $p(x_{(k)}) = 0$. For $k = K+1,\dotsc,n$,
    using the inequality $\abs{q(x)}\leq 1$ for $x\in[-1,1]$ and \eqref{compare_y}, we obtain
    $$
    \abs{p(x_{(k)})} \leq \frac{1}{2}\prod_{j=0}^{K}\frac{x_{(k)}-x_{(j)}}{x_{(k)}-y_j} \leq \frac{1}{2}\frac{x_{(k)}-x_{(0)}}{x_{(k)}-y_0}\leq 1.
    $$
    Here, we use
    $$
    \frac{x_{(k)}-x_{(0)}}{x_{(k)}-y_0} \leq \frac{y_K+1}{y_K-y_0} = 1+ \frac{1+y_0}{y_K-y_0} = 1+ \frac{\sin^2{(\pi/(4m))}}{\sin{((K+1)\pi/(2m))\sin{(K\pi/(2m))}}} \leq 1 + \frac{\pi^2}{16K(K+1)} \leq 2
    $$
    where we apply the inequalities $2t/\pi \leq \sin{(t)} \leq t$ for $0 \leq t \leq \pi/2$. Therefore, we have established \eqref{eq:pxj}.
    
    To estimate $\|p\|_\infty$, we evaluate $p$ at the point
    $$
    x^* =  -\cos{\frac{\pi}{m}},
    $$
    which lies midway between $y_0$ and $y_1$. It is evident that $\abs{q(x^*)} = 1$. Hence, we obtain that
    \begin{equation*}
    \begin{aligned}
        \|p\|_\infty 
         & \geq |p(x^*)| 
         = \frac{1}{2}|q(x^*)|\prod_{j=0}^{K}\frac{|x^*-x_{(j)}|}{|x^*-y_j|}  \\
        & = \frac{1}{2}\frac{x^*-(-1)}{x^*-y_0}\prod_{j=1}^{K}\frac{x_{(j)}-x^*}{y_j-x^*}  \overset{(a)}{\geq} \frac{1}{2}\prod_{j=1}^{K}\frac{x_{(j)}+1}{y_j+1}  \overset{(b)}{\geq} \frac{1}{2}\prod_{j=1}^{K}\frac{\left(\frac{j}{n^{1+s}}\right)^{\frac{1}{1+\beta}}}{\frac{2\pi^2j^2}{m^2}}  \\
        & \geq \frac{1}{2}\prod_{j=1}^{K}\left(\frac{j^{\frac{1}{1+\beta}}}{j^2} \left(\frac{1}{(2\pi^2)^{\frac{1+\beta}{1+2\beta}}}\left(\frac{m^{2(1+\beta)}}{n^{1+s}}\right)^{\frac{1}{1+2\beta}}\right)^{\frac{1+2\beta}{1+\beta}}\right) \\
        & \geq \frac{1}{2}\prod_{j=1}^{K}{\left(\frac{K}{j}\right)^{\frac{1+2\beta}{1+\beta}}}  = \frac{1}{2}\left(\frac{K^K}{K!}\right)^{\frac{1+2\beta}{1+\beta}}  \\
        & \overset{(c)}{\geq} \frac{1}{2}\left(\frac{e^K}{e^{\frac{1}{12}}\sqrt{2\pi K}}\right)^{\frac{1+2\beta}{1+\beta}}  \geq \eta_1\frac{\left(e^{\frac{1+2\beta}{1+\beta}}\right)^K}{K^{\frac{1+2\beta}{2(1+\beta)}}} 
         \geq \eta_2\nu^{\left(\frac{m^{2(1+\beta)}}{n^{1+s}}\right)^{\frac{1}{1+2\beta}}}, 
    \end{aligned}
    \end{equation*}
    with some $\eta_1, \eta_2>0$ and $\nu>1$ only depending on $\gamma$. Here, the inequality (a) is derived from \eqref{compare_y}. The inequality (b) follows directly from \eqref{compare_i} and \eqref{yj_ineq}. And the inequality (c) is obtained by invoking the inequality
    $$
    \frac{k^k}{k!} > \frac{e^k}{e^{\frac{1}{12}}\sqrt{2\pi k}}, \quad \text{for} \quad k\geq 1,
    $$
    implied by \eqref{Stirling}.

    Observe that  \eqref{compare_i} is the precondition we have assumed.
    It is worth noting that Lemma \ref{prob_estimate} provides an estimation for the probability of equation \eqref{compare_i} being valid, which is not less than
     $1-\frac{2e^2\bar{c}_1}{n^s}$ where $\bar{c}_1 = \frac{c_2 2^\alpha}{1+\beta}$. Consequently, in the context of Case \romannumeral1, we can infer that
    $$
    B(n,m)
    \geq \eta_2\nu^{\left(\frac{m^{2(1+\beta)}}{n^{1+s}}\right)^{\frac{1}{1+2\beta}}}
    $$
    holds true with a probability of at least $1-\frac{2e^2\bar{c}_1}{n^s}$.
    
    \textbf{Case \romannumeral2}. For the scenario where $\left(\frac{m^{2(1+\beta)}}{n^{1+s}}\right)^{\frac{1}{1+2\beta}} \leq \hat{c}$, let's assign $\eta_3 = \nu^{-\hat{c}}$.  Consequently, we have
    $$
    B(n,m) 
    \geq 1
    \geq \eta_3 \nu^{\left(\frac{m^{2(1+\beta)}}{n^{1+s}}\right)^{\frac{1}{1+2\beta}}}.
    $$
By combining Case \romannumeral1\,\, and Case \romannumeral2 \,\, with $\eta = \min{\{\eta_2, \eta_3\}}$, we can conclude the proof.
\end{proof}


\section{The proof of Theorem \ref{opt_rates} and Corollary \ref{opt_sam_rat}}
\label{4}

In this section, we present the proof of Theorem \ref{opt_rates} and Corollary \ref{opt_sam_rat} which provides a positive answer to Question \uppercase\expandafter{\romannumeral1}. 

\begin{proof}[Proof of Theorem \ref{opt_rates}]
    Lemma \ref{eq} tells us that 
    \begin{equation}
        \label{lem23}
        \kappa_{2}\left(P_m^n\right) = D(n,m) \quad \text{almost surely}.
    \end{equation}
    And by Corollary \ref{L2_behavior}, we see that there exist constants $\eta_1>0$, $\nu_1>1$ depending on $\alpha, \beta$, such that for any $s \in \mathbb{R}_{+}$ andall $1 \leq m < n$ with $n>(2e^2\bar{c})^{1/s}$ where $\bar{c} = \frac{c \cdot 2^{\min\{\alpha, \beta\}}}{1+\gamma}$ with the constant $c$ from \eqref{Jacobi_e}, we have 
    $$
    D(n,m)
    \geq \frac{\eta_1}{m^{1+\gamma}} \nu_1^{\lambda_1}, \quad \lambda_1 = \left(\frac{m^{2(1+\gamma)}}{n^{1+s}}\right)^{\frac{1}{1+2\gamma}}, 
    $$
    with probability at least $1-\frac{2e^2\bar{c}}{n^s}$. Under the setting $n \asymp m^{1/\tau}$, there exist positive constants $c_1, c_2$ and a positive integer $n_0$, such that $c_1 n^\tau \leq m \leq c_2 n^\tau$ holds for all $n \geq n_0$. Hence, for any $s \in \mathbb{R}_{+}$ satisfying $s<2(1+\gamma)\tau-1$ and all $n \geq n_1:=\max\{n_0, (2e^2\bar{c})^{1/s}\}$, we have 
    \begin{equation}
        \label{cor_app}
        D(n,m) \geq \frac{\eta_1}{c_2^{1+\gamma}n^{(1+\gamma)\tau}} \nu_1^{\lambda_2}, \quad \lambda_2 = \left(c_1^{2(1+\gamma)}n^{2(1+\gamma)\tau-1-s}\right)^{\frac{1}{1+2\gamma}}, 
    \end{equation}
    with probability at least $1-\frac{2e^2\bar{c}}{n^s}$. Then we set $\eta = \eta_1/c_2^{1+\gamma}$, $\nu = \nu_1^{c_1^{2(1+\gamma)/(1+2\gamma)}}>1$. Combining \eqref{lem23} and \eqref{cor_app}, we arrive at the desired conclusion that 
    $$
    \kappa_{2}\left(P_m^n\right) \geq \frac{\eta}{n^{(1+\gamma)\tau}} \nu^{n^r}, \quad r = \frac{2(1+\gamma)\tau -1 -s}{1+2\gamma},
    $$
    for all $n \geq n_1$, with probability at least $1-\frac{2e^2\bar{c}}{n^s}$.
\end{proof}

\begin{proof}[Proof of Corollary \ref{opt_sam_rat}]
    Using Lemma \ref{eq} again, we obtain that $\kappa_{2}\left(P_m^n\right) 
    = D(n,m)$ almost surely. And Corollary \ref{bound_iff_cond} indicates that, when $\gamma = \max\{\alpha, \beta\} > -1/2$, the necessary and sufficient condition for $D(n,m) \lesssim 1$ with high probability is 
    $$
    n \asymp m^{2(1+\gamma)}
    $$
    disregarding the logarithmic factor of $n$. By combining these two findings, we complete the proof.
\end{proof}

\section{The proof of Theorem \ref{impossibility_3}}
\label{5}

In this section, we focus on Question II concerning the probability measure with respect to the modified Jacobi weight functions (see Section \ref{Jacobi}). We present the proof of Theorem \ref{impossibility_3}.
\begin{proof}[Proof of Theorem \ref{impossibility_3}]
    Define the occurrence of the inequalities \eqref{undefine} and \eqref{impos_result} as events A and B respectively. We first consider the case where $P(A) = P_0 = 1$.

    Let $E_\rho$  denote the Bernstein ellipse containing $E$ with a parameter $\rho>1$.     
    From \eqref{undefine}, we derive the following inequality:
    \begin{equation}
        \label{undefine_accuracy_2}
        \left\|f-\phi_n(f)\right\|_{\infty} \leq \frac{1}{2} \rho^{-\mu n^\tau}\|f\|_{E_\rho}, \quad n_0 \leq n < \infty, 
    \end{equation}
    for some sufficiently large $n_0$ depending on $M$, $\sigma$, $\tau$, and for a small constant $\mu>0$ depending on $\sigma$, $\rho$, provided $f\in B(E)$. Let $p_d \in \mathbb{P}_d$. Since $p_d \in B(E)$,
     we can substitute $p_d$ for $f$ in \eqref{undefine_accuracy_2}. 
      Additionally, a well-known result by Bernstein \cite{bernstein1912ordre} affirms that $\|p_d\|_{E_\rho} \leq \rho^d\|p_d\|_\infty$. Hence
    \begin{equation} \label{undefine_change}
    \begin{aligned}
        \|\phi_n(p_d)\|_\infty 
        & \geq \|p_d\|_\infty - \|p_d-\phi_n(p_d)\|_\infty \\
        & \geq \|p_d\|_\infty - \frac{1}{2} \rho^{-\mu n^\tau}\|p_d\|_{E_\rho}  \\
        & \geq \|p_d\|_\infty - \frac{1}{2} \rho^{d-\mu n^\tau}\|p_d\|_\infty  \\
        & \geq \frac{1}{2}\|p_d\|_\infty,
    \end{aligned}
    \end{equation}
    whenever $\rho^{d-\mu n^\tau} \leq 1$, i.e., $d \leq \lfloor \mu n^\tau \rfloor$, with $n \geq n_0$.
        
    Notice that \eqref{undefine} implies $\phi_n(0) = 0$. 
    In the definition of the \emph{condition number} $\kappa_\infty$ in \eqref{cond_inf}, we now consider $f = 0$ and $h = \delta\cdot p_d$ for any $\delta>0$.
    Subsequently, for all $n \geq n_0$, the following holds:
    \begin{equation}\label{11}
    \begin{aligned}
        \kappa_\infty(\phi_n) 
        & \geq \lim_{\delta \rightarrow 0{+}} \sup_{\substack{p_d \in \mathbb{P}_d, d \leq \lfloor \mu n^\tau \rfloor \\ \|p_d\|_{n, \infty} \neq 0}}{\frac{\|\phi_n(\delta\cdot p_d) - \phi_n(0)\|_\infty}{\|\delta\cdot p_d\|_{n,\infty}}}  \\
        & \stackrel{(a)}{\geq} \lim_{\delta \rightarrow 0{+}} \sup_{\substack{p_d \in \mathbb{P}_d, d \leq \lfloor \mu n^\tau \rfloor \\ \|p_d\|_{n, \infty} \neq 0}}{\frac{\|\delta \cdot p_d\|_\infty}{2\|\delta\cdot p_d\|_{n,\infty}}}  \\
        & = \sup_{\substack{p_d \in \mathbb{P}_d, d \leq \lfloor \mu n^\tau \rfloor \\ \|p_d\|_{n, \infty} \neq 0}}{\frac{\|p_d\|_\infty}{2\|p_d\|_{n,\infty}}}
        = \frac{1}{2}B(n,\lfloor \mu n^\tau \rfloor).
    \end{aligned}
    \end{equation}
    Here, the inequality $(a)$ is derived from \eqref{undefine_change}.
    Since $\gamma > -1/2$, Lemma \ref{polynomials_random} implies that there exist constants $\eta_1>0$, $\nu_1>1$ depending on $\gamma$, such that for any $s \in \mathbb{R}_{+}$ satisfying $s<2(1+\gamma)\tau-1$ and all $n \geq n_1:=\max\{n_0,(2e^2\bar{c})^{1/s}\}$ where $\bar{c} = \frac{c_2 2^{\min\{\alpha, \beta\}}}{1+\gamma}$,
    $$
    B(n,\lfloor \mu n^\tau \rfloor)
    \geq \eta_1 \nu_1^{\lambda_1}, \quad \lambda_1 = \left(\frac{\lfloor \mu n^\tau \rfloor^{2(1+\gamma)}}{n^{1+s}}\right)^{\frac{1}{1+2\gamma}},
    $$
    with probability at least $1-\frac{2e^2\bar{c}}{n^s}$. Since $\lfloor \mu n^\tau \rfloor > \frac{\mu}{2}n^\tau$ for $n \geq (2/\mu)^{1/\tau}$, it follows that
    \begin{equation}
        \label{22}
        B(n,\lfloor \mu n^\tau \rfloor)
        \geq \eta_1 \nu_1^{\lambda_2}, \quad \lambda_2 = \left((\mu/2)^{2(1+\gamma)}\,n^{2(1+\gamma)\tau-1-s}\right)^{\frac{1}{1+2\gamma}},
    \end{equation}
    for any $s \in \mathbb{R}_{+}$ satisfying $s<2(1+\gamma)\tau-1$ and all $n \geq n_2:=\max\{n_1,(2/\mu)^{1/\tau}\}$, with probability at least $1-\frac{2e^2\bar{c}}{n^s}$. Then we set $\eta=\eta_1/2$, $\nu = \nu_1^{(\mu/2)^{2(1+\gamma)/(1+2\gamma)}}>1$. Consequently, combining \eqref{11} with \eqref{22}, we obtain that \eqref{impos_result} holds for all $n \geq n_2$ with probability at least $1-\frac{2e^2\bar{c}}{n^s}$, i.e., 
    $$
    P(B|A)\geq 1-\frac{2e^2\bar{c}}{n^s}.    
    $$
    Therefore 
    $
    P(B) \geq P(AB)  = P(A)P(B|A) \geq 1-\frac{2e^2\bar{c}}{n^s}.
    $

    For the case where $P(A) = P_0 < 1$, we observe that
    $
    P(B) \geq P(AB) = P(A)P(B|A) \geq \left(1-\frac{2e^2\bar{c}}{n^s}\right)P_0,
    $
    thereby completing the proof.
\end{proof}

\section{Numerical Illustration}
\label{6}

In this section, we aim to present the relationship between the stability of $P_m^n$ measured by $\kappa_2(P_m^n)$ and the random sampling rates through  numerical experimental approach, and further illustrate our results. 

Since it has been stated in Lemma \ref{eq} that $\kappa_2(P_m^n) = D(n,m) = \left(\lambda_{\min}(\mathbf{G})\right)^{-1/2}$ almost surely where the matrix $\mathbf{G}$ is given by $\mathbf{G}=(\left<L_j, L_k\right>_n)_{j,k=1,\dotsc,m+1}$ with $\{L_j\}_{j=1}^{m+1}$ an orthonormal basis of $\mathbb{P}_m$ in the sense of $L^2(X,\rho_X)$, we calculate the values of $\left(\lambda_{\min}(\mathbf{G})\right)^{-1/2}$ under different combinations of $m$ and $n$ with $m<n$. In practical computations, when $\mathbf{G}$ is nearly singular to the extent that the computation becomes imprecise, we proceed by equating $\lambda_{\min}(\mathbf{G}) = 1 \times 10^{-13}$.

Three choices of the sampling measure $\rho_X$ based on the \emph{Jacobi} weight functions, which are characterized by the parameters $\alpha$ and $\beta$, are tested. Figure \ref{test} shows $\log_{10}{\kappa_2(P_m^n)}$ as $n$ and $m$ range from $1$ to $100$, with $m<n$. We take the average value of $\kappa_2(P_m^n)$ obtained from $100$ repeated random samplings. Points on the dashed line follow the sampling rate $n \asymp m^{2(1+\gamma)}$ where $\gamma = \max\{\alpha,\beta\}$. As seen in Figure \ref{test}, $\kappa_2(P_m^n) \lesssim 1$ can be maintained with high probability if the sampling rate is at least $n \asymp m^{2(1+\gamma)}$ (corresponding to the left side of the dashed line), otherwise $\kappa_2(P_m^n)$ grows exponentially, which is in accordance with our theory.

\begin{figure}[htb]
    \centering
    \includegraphics[width=0.32\textwidth]{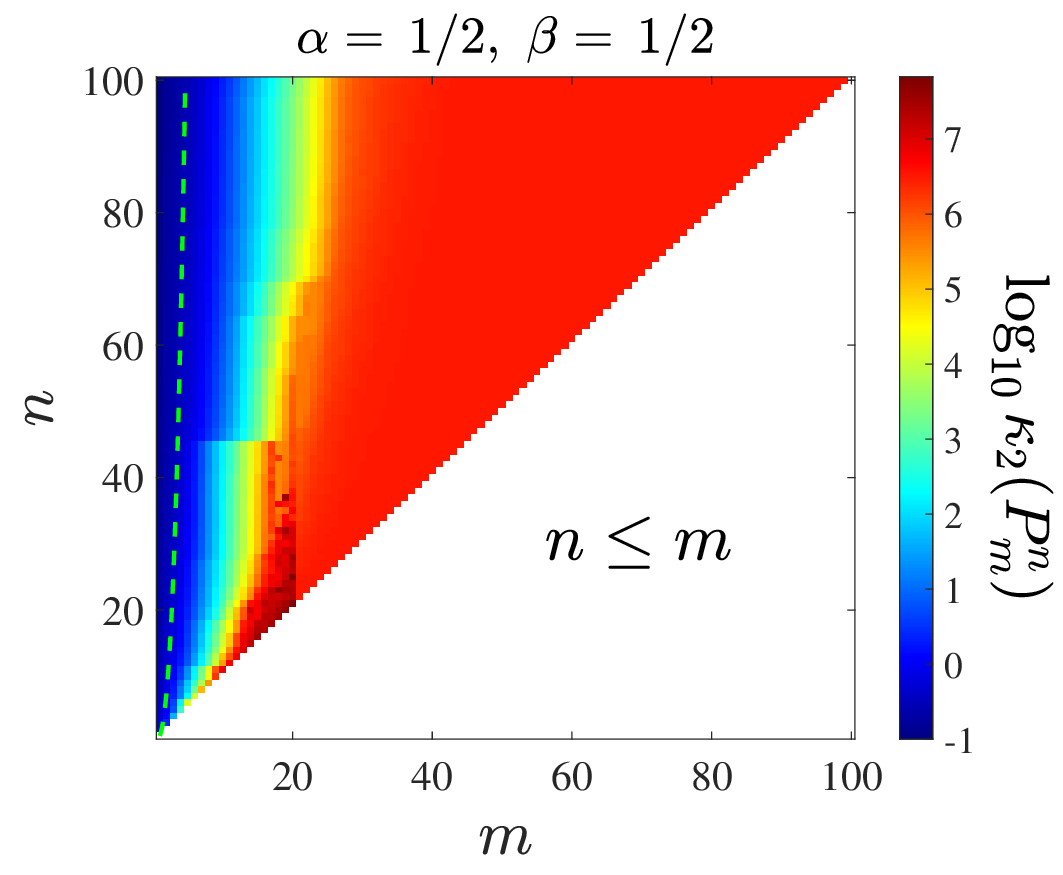} 
    \hspace{0.5mm}
    \includegraphics[width=0.32\textwidth]{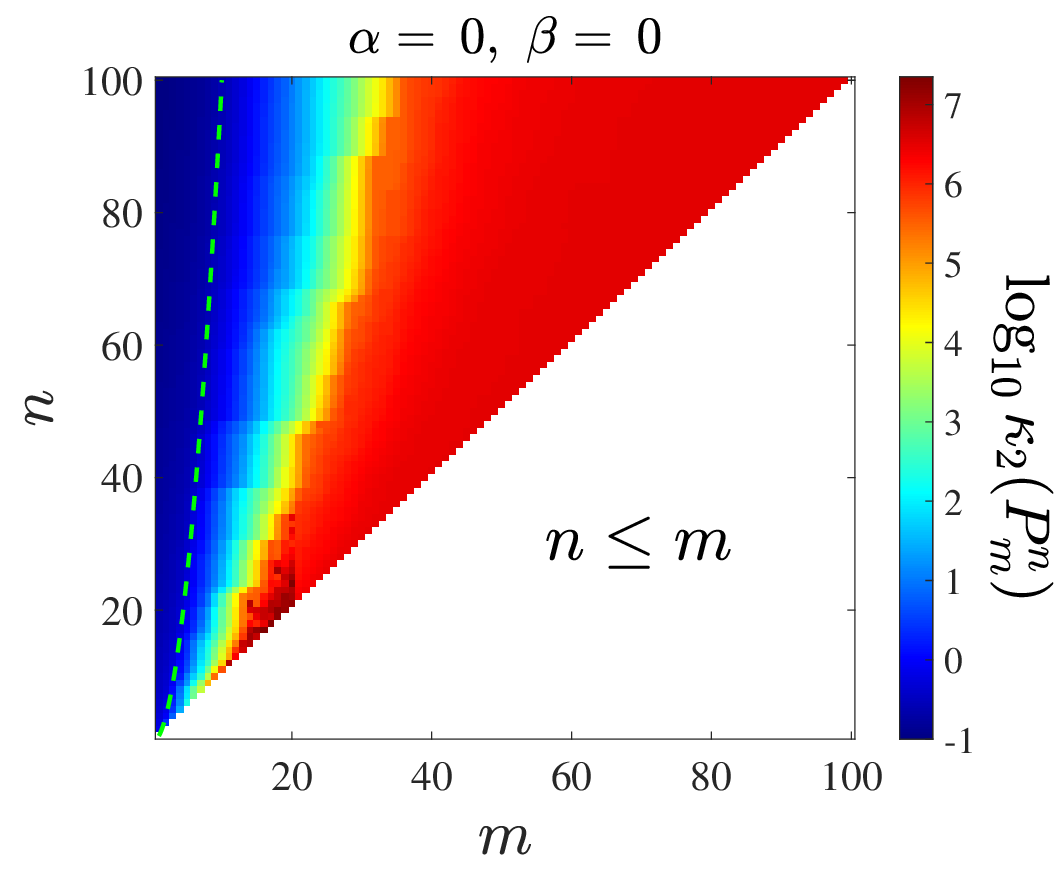} 
    \hspace{0.5mm}
    \includegraphics[width=0.32\textwidth]{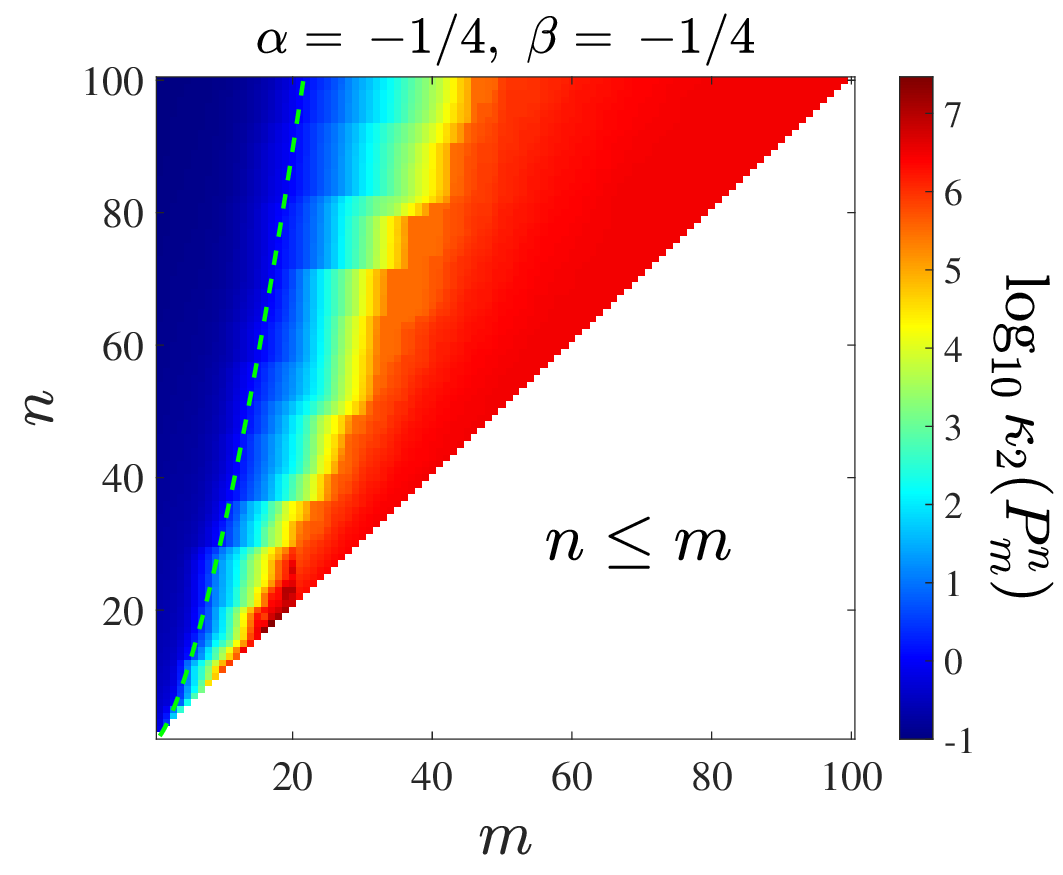} 
    \caption{Values of $\log_{10}{\kappa_2(P_m^n)}$ for three choices of the sampling measure $\rho_X$ with $\alpha=\beta=1/2$ (left), $\alpha=\beta=0$ (middle), $\alpha=\beta=-1/4$ (right). Dashed lines follow $n = m^{3}$ (left), $n = m^{2}$ (middle), $n = m^{3/2}$ (right) respectively.}
    \label{test}
\end{figure}




\makeatletter
\renewcommand\@biblabel[1]{#1.}
\makeatother

\end{document}